\documentclass[review]{elsarticle}
\usepackage{fullpage, amsthm, amsmath, amsfonts, amssymb, url, mathrsfs, stmaryrd}
\usepackage[colorlinks]{hyperref}
\usepackage{setspace}

\numberwithin{equation}{section}

\newcommand{\R}{\mathbb R}

\newcommand{\N}{\mathbb N}

\newcommand{\M}{\mathcal M}

\newcommand{\ra}{\rangle}
\newcommand{\la}{\langle}

\newcommand{\energ}{\dot H^1 \times L^2}
\newcommand{\EW}{E(W,0)}
\newcommand{\nW}{\nabla W}

\newtheorem{lem}{Lemma}[section]
\newtheorem{thm}[lem]{Theorem}
\newtheorem{ppn}[lem]{Proposition}
\newtheorem{defn}[lem]{Definition}
\newtheorem{clm}[lem]{Claim}
\newtheorem{cor}[lem]{Corollary}
\newtheorem{rmk}[lem]{Remark}
\newtheorem{note}[lem]{Notation}
\newtheorem{conj}[lem]{Conjecture}

\newcommand{\de}{\delta}
\newcommand{\e}{\epsilon}

\newcommand{\om}{\omega}
\newcommand{\lam}{\lambda}

\newcommand{\sig}{\sigma}

\newcommand{\eps}{\epsilon}

\newcommand{\wa}{\rightharpoonup}
\newcommand{\rar}{\rightarrow}

\newcommand{\p}{\partial}

\newcommand{\supp}{\operatorname{supp}}

\newcommand{\vl}{\vec \ell}

\begin{document}

\begin{frontmatter}

\title{Concentration-compactness and universal profiles \\ for the non--radial energy critical wave equation}

\author{Thomas Duyckaerts \fnref{myfootnote1}}
\address{LAGA (UMR 7539), Universit\'e Paris 13, Sorbonne Paris Cit\'e}
\fntext[myfootnote1]{Partially supported by ANR JCJC Grant Scheq, ERC Grant Dispeq and ERC advanced grant 
no. 291214, BLOWDISOL.}

\author{Carlos Kenig\fnref{myfootnote2}}
\address{University of Chicago}
\fntext[myfootnote2]{Partially supported by NSF grants DMS--1265429 and DMS--1463746.}

\author{Frank Merle\fnref{myfootnote3}}
\address{Cergy--Pontoise (UMR 8088), IHES}
\fntext[myfootnote3]{Partially supported by ERC advanced grant no. 291214, BLOWDISOL.}

\author{To Juan Luis V\'azquez, on his 70th birthday. }

\begin{abstract}
In this paper, we give an overview of the authors' work on applications of the method of concentration--compactness to 
global well--posedness, scattering, blow--up and universal profiles for the energy critical wave equation in the non--radial setting. New results and proofs are also given.
\end{abstract}


\end{frontmatter}

\section{Introduction}

The main purpose of this paper is to give an overview of the recent developments on the study of the focusing energy critical wave equation, using concentration--compactness and rigidity theorems to study global well--posedness, scattering, blow--up, and universal profiles. Our final goal in this direction is to prove the soliton resolution conjecture, described below as Conjecture \ref{conj55}.  Thus, 
we will consider, for an interval $I$, $0 \in I$, 
\begin{align}\label{nlw}
\left\{
     \begin{array}{lr}
       \partial_t^2 u - \Delta u - |u|^{\frac{4}{N-2}} u = 0, \quad (t,x) \in I \times \R^N \\
       u|_{t = 0} = u_0 \in \dot H^1, \quad \partial_t u|_{t = 0} = u_1 \in L^2,
     \end{array}
   \right.,
\end{align}
where $u$ is real valued, $N \in \{3,4,5\}$, $\dot H^1 = \dot H^1(\R^N)$, and $L^2 = L^2(\R^N)$. 

The Cauchy problem is locally well--posed in $\dot H^1 \times L^2$.  This space is invariant under the scaling of the equation: 
if $u$ is a solution to \eqref{nlw}, $\lambda > 0$, and $u_\lambda(t,x) = \lambda^{-(N-2)/2} u(t/\lambda, x/\lambda)$, then $u_\lambda$ is also a solution and $\| u_\lambda(0) \|_{\dot H^1} = \| u_0 \|_{\dot H^1}$, $\| \partial_t u_\lambda(0) \|_{L^2} = 
\| u_1 \|_{L^2}$. 

The energy 
\begin{align}\label{energy}
E(u(t),\partial_t u(t)) = \frac{1}{2} \int |\partial_t u(t,x)|^2 dx + \frac{1}{2} \int |\nabla u(t,x)|^2 dx - \frac{N-2}{2N} 
\int |u(t,x)|^{\frac{2N}{N-2}} dx,
\end{align}
and the momentum 
\begin{align}\label{mom}
P(u(t),\partial_t u(t)) = \int \partial_t u(t,x) \nabla u(t,x) dx,
\end{align}
are independent of $t$ and also invariant under the scaling. 

It is well--known, from the work of Levine \cite{28} that solutions to \eqref{nlw} may blow--up in finite time.  Explicit blow--up solutions, for which the norm in $\dot H^1 \times L^2$ blows--up as one approaches the maximal forward time of definition $T_+$ of the solution $u$, 
can be easily constructed using the ODE associated to \eqref{nlw} and finite speed of propagation.  Moreover, the local well--posedness theory does not rule out type II solutions, i.e. solutions such that 
\begin{align}\label{sec14}
\sup_{t \in [0,T_+)} \| (u(t),\partial_t u(t)) \|_{\dot H^1 \times L^2} < \infty, 
\end{align}
which have $T_+ < \infty$ (type II blow--up solutions) or for which $T_+ = \infty$ and the solution does not scatter (see the comments 
after \eqref{sec23}).  The finite time blow--up examples are due to Krieger--Schlag--Tataru \cite{27}, Krieger--Schlag \cite{26}, Hillairet--Rapha\"el \cite{26}, and Jendrej \cite{20} while global solutions which do not scatter are solutions to the associated elliptic equation, 
$Q \in \dot H^1$, 
\begin{align}\label{sec15}
\Delta Q + |Q|^{\frac{4}{N-2}} Q = 0, 
\end{align}
such as the \lq ground state'
\begin{align}\label{sec16}
W(x) = \left ( 1 + \frac{|x|^2}{N(N-2)} \right )^{-\frac{N-2}{2}},
\end{align}
and more complicated ones such as those constructed by Donninger--Krieger \cite{8}.  This shows the richness of the class of 
type II solutions of \eqref{nlw}.  The soliton resolution conjecture (Conjecture \ref{conj55} ) is an attempt to give a description of all 
type II solutions. 

The concentration--compactness method was first developed for elliptic equations, through the works of P.L. Lions, P. G\'erard, 
H. Br\'ezis, J.M. Coron and others.  Of particular importance for us is the profile decomposition and its associated Pythagorean expansions, which are consequences of the orthogonality of the parameters, and which were introduced by Br\'ezis--Coron in 
\cite{2}.  See the bibliography of  \cite{24} and \cite{10} for precise references.  For the wave equation, the profile decomposition was first introduced by Bahouri--G\'erard \cite{1}.  In connection with the wave equation, there is s subtlety in the Pythagorean expansions, which we pointed out recently in \cite{13}, namely that the natural expansions \eqref{sec39} and \eqref{sec310} (which did not appear in \cite{1}) which we used without proof in several of our works, are in fact false.  This leaves a gap in some of our proofs.  In \cite{13} we 
addressed these gaps in the radial case.  We take this opportunity to address the gaps in the non--radial setting, occurring in \cite{24} and \cite{10}, as will be explained in detail in Sections 4,7, and 8. 
 
In Section 2 we give some background material.  We briefly review the local theory of the Cauchy problem \eqref{nlw} and some results on the associated nonlinear elliptic equation \eqref{sec15}.  We also review some elliptic--type variational estimates related
to the ground state \eqref{sec16}.  In Section 3, we discuss the Bahouri--G\'erard \cite{1} profile decomposition and its use to study 
\eqref{nlw}.  We give some precisions on this, and in particular, prove some correct variants of the false expansions \eqref{sec39}
and \eqref{sec310}.  In Section 4, we discuss in detail the concentration--compactness part of \cite{24}, which was mostly omitted in 
\cite{24} and give complete proofs of the construction and compactness of the critical elements, without using \eqref{sec39} and 
\eqref{sec310}.  We also formulate an alternative version of the main result in \cite{24} and show that, surprisingly, it is actually equivalent 
to the main result in \cite{24}.  In Section 5, we give a comprehensive review of the main known results on compact solutions to \eqref{nlw} (see Definition \ref{defn52}).   These solutions are fundamental for the implementation of the concentration--compactness method introduced in \cite{23}, \cite{24}, and \cite{25}.  In Section 6 we describe a different formulation of the concentration--compactness method (from \cite{15}) in which we show that all type II solutions (see \eqref{sec14}), up to modulation, converge weakly to a compact solution.  This formulation applies to a wide range of dispersive problems.  Sharper results are then shown for \eqref{nlw}, using the results in Section 5 on compact solutions. In Section 7 we use the new concentration--compactness formulation to prove a generalization of the main result in \cite{24}, formulated in \cite{10}, without using critical elements, which are bypassed by the results explained in Section 6.  We also prove sharp versions of
concentration results for type II blow--up solutions, first discussed in \cite{24} and \cite{10}.  In Section 8 we study universal profiles for type II 
blow--up solutions, near $W$, using the results in Sections 6 and 7, filling a gap in some of the proofs in \cite{10}, when $N = 5$, caused by the failure of \eqref{sec39} and \eqref{sec310}.  

Notation: $\| \cdot \|$ will always denote the $L^2$ norm.  We assume throughout that $N \in \{ 3,4,5 \}$ unless otherwise mentioned.  If $u$ is a function of $(t,x)$, then we denote $\vec u(t,x) = (u(t,x), \p_t u(t,x))$.

\section{Preliminaries}

The Cauchy problem for equation \eqref{nlw} was developed in the 80's and 90's.  See \cite{10}, Preliminaries, for instance, for the 
relevant references.  If $I$ is an interval, we denote 
\begin{align*}
S(I) &= L^{\frac{2(N+1)}{N-2}}(I \times \R^N), \\
W(I) &= L^{\frac{2(N+1)}{N-1}}(I \times \R^N).  
\end{align*}
Let $S_L(t)$ be the propagator associated to the linear wave equation.  By definition, if $(v_0,v_1) \in \energ$ and $t \in \R$, then 
$v(t) = S_L(t)(v_0,v_1)$ is the solution of 
\begin{align}\label{lw}
\left\{
     \begin{array}{lr}
       \partial_t^2 v - \Delta v = 0, \quad (t,x) \in \R \times \R^N, \\
       v|_{t = 0} = v_0, \quad \partial_t v|_{t = 0} = v_1,
     \end{array}
   \right..
\end{align}

We have 
\begin{align*}
S_L(t)(v_0,v_1) = \cos(t \sqrt{-\Delta}) v_0  + \frac{1}{\sqrt{-\Delta}} \sin(t \sqrt{-\Delta}) v_1. 
\end{align*}
By Strichartz and Sobolev estimates (see the references in \cite{10}), 
\begin{align*}
\| v \|_{S(\R)} + \| D_x^{1/2} v \|_{W(\R)} + \| D_x^{-1/2} \partial_t v \|_{W(\R)} \leq C \| (v_0,v_1) \|_{\energ}.
\end{align*}
A solution of \eqref{nlw} on an interval $I$, where $0 \in I$, is a function $u \in C(I; \dot H^1)$ such that $\partial_t u \in C(I; L^2)$, 
\begin{align*}
\forall J \Subset I,\;
\| u \|_{S(J)} + \| D_x^{1/2} u \|_{W(J)} + \| D^{-1/2}_x \partial_t u \|_{W(J)} < \infty,
\end{align*}
and $u$ satisfies the Duhamel formulation 
\begin{align*}
u(t) = S_L(t)(u_0,u_1) + \int_0^t \frac{\sin((t-s)\sqrt{-\Delta})}{\sqrt{-\Delta}} |u(s)|^{\frac{4}{N-2}}u(s) ds. 
\end{align*}
We recall that for any initial condition $(u_0,u_1) \in \energ$, there is a unique solution $u$ defined on a maximal interval of definition 
$I_{\max}(u) = (T_-(u),T_+(u))$.  Furthermore, $u$ satisfies the blow--up criterion
\begin{align}\label{sec23}
T_+(u) < \infty \implies \| u \|_{S(0,T_+(u))} = \infty. 
\end{align}
As a consequence, if $\| u \|_{S(0,T_+(u))} < \infty$, then $T_+(u) = \infty$.  Moreover, in this case, the solution scatters forward in time in $\energ$: there exists a solution $v_L$ of the linear equation \eqref{lw} such that 
\begin{align*}
\lim_{t \rightarrow \infty} \| \vec u(t) - \vec v_L(t) \|_{\energ} = 0. 
\end{align*}
Of course, an analogous statement holds backward in time also.

If $\| S_L(\cdot)(u_0,u_1) \|_{S(I)} = \delta < \delta_0$, for some small $\delta_0$, then $u$ is defined on $I$ and close to the linear solution with initial condition $(u_0,u_1)$: if $A = \| D^{1/2}_x S_L(\cdot) (u_0,u_1) \|_{W(I)}$, we have 
\begin{align*}
\| u(\cdot) - S_L(\cdot)(u_0,u_1) \|_{S(I)} + \sup_{t \in I} \| \vec u(t) - \vec S_L(t)(u_0,u_1) \|_{\energ} \leq C A \delta^{\frac{4}{N-2}}. 
\end{align*}
We next turn to the associated elliptic equation: $Q \in \dot H^1$, 
\begin{align}\label{sec24}
\Delta Q + |Q|^{\frac{4}{N-2}} Q = 0.
\end{align}
Clearly, time independent solutions of \eqref{nlw} are solutions to \eqref{sec24} and conversely. There is an explicit solution of 
\eqref{sec24}, namely the ground state 
\begin{align*}
W(x) = \left ( 1 + \frac{|x|^2}{N(N-2)} \right )^{-\frac{N-2}{2}}.
\end{align*}
$W$ is, up to translation and scaling, the only nonnegative nontrivial solution of \eqref{sec24}, i.e. $W_{\lambda_0}(x + x_0)
= \lambda_0^{-(N-2/2} W((x + x_0)/\lambda_0)$ are the only nonnegative nontrivial solutions of \eqref{sec24}.   This is a well--known result of Gidas--Ni--Nirenberg.  $W$ has a characterization as the extremal in the Sobolev embedding: 
\begin{align}\label{sec25}
\| f \|_{L^{\frac{2N}{N-2}}} \leq C_N \| \nabla f \|_{L^2}, 
\end{align}
where $C_N$ denotes the best constant.  Aubin and Talenti characterized $W$ as the unique extremizer to \eqref{sec25}, modulo a multiplicative constant,
translation and scaling.  Since by \eqref{sec24} $\int |\nabla W|^2 dx = \int |W|^{\frac{2N}{N-2}} dx$ and 
$C_N \| \nabla W \|_{L^2} = \| W \|_{L^{\frac{2N}{N-2}}}$, we have $\int |\nabla W|^2 dx = C_N^{-N}$ and 
\begin{align}\label{sec26}
E(W,0) = \frac{1}{2} \int |\nabla W|^2 dx - \frac{N-2}{2N} \int |W|^{\frac{2N}{N-2}}dx = \frac{1}{NC_N^N} = 
\frac{1}{N} \int |\nabla W|^2 dx. 
\end{align}
Another important fact about $W$ is: $\pm W_{\lambda_0}$ are the only radial solutions of \eqref{sec24}.  This follows from a combination of results of Pohozaev, Gidas--Ni--Nirenberg, and Kwong.  For the precise references, see \cite{22}.  We now turn 
to general solutions to \eqref{sec24}.  We recall 
\begin{thm}\label{thm21}
If $Q$ is a nonzero solution of \eqref{sec24} such that $\int |\nabla Q|^2 dx < 2 \int |\nabla W|^2 dx$, then 
$Q(x) = \pm W_{\lambda_0}(x+x_0)$ for some $\lambda_0 > 0$ and $x_0 \in \R^N$. 
\end{thm}
This is a well--known result, for a proof see Lemma 2.6 in \cite{10}. 

\begin{defn}\label{defn22}
We define the set 
\begin{align*}
\Sigma = \left \{ Q \in \dot H^1 : Q \mbox{ solves \eqref{sec24} and } Q \neq 0 \right \}.
\end{align*}
\end{defn}
We turn to the invariances of \eqref{sec24}.  If $Q \in \Sigma$, then $x \mapsto Q(x + b)$ where $b \in \R^N$, $x \mapsto
Q(Px)$ where $P \in O_N$, the orthogonal group, $x \mapsto \lambda^{(N-2)/2} Q(\lambda x)$ where $\lambda > 0$, and 
$x \mapsto |x|^{2-N} Q ( x / |x|^2 )$ are also in $\Sigma$.  We denote by $\M$ the group of isometries of 
$L^{\frac{2N}{N-2}}$ (and $\dot H^1$) generated by the preceding transformations.  $\M$ defines a $N'$--parameter family of transformations in a neighborhood of the identity, where $N' = 2N+1 + N(N-1)/2$ (see \cite{10}).  If $Q \in \Sigma$, we let 
\begin{align}\label{sec28}
L_Q = -\Delta - \frac{N+2}{N-2} |Q|^{\frac{4}{N-2}}
\end{align}
be the linearized operator around $Q$. Let $\mathcal Z_Q = \left \{ f \in \dot H^1 : L_Q f = 0 \right \}$ and 
\begin{align*}
\tilde{\mathcal Z}_Q = \mbox{span} \Bigl \{ (2-N)x_j Q + |x|^2  &  \partial_{x_j} Q - 2 x_j x \cdot \nabla Q, \partial_{x_j} Q, 1 \leq j
\leq N,  \\  &  x_j \partial_{x_k} Q - x_k \partial_{x_j} Q, 1 \leq j < k \leq N, \frac{N-2}{2} Q + x \cdot \nabla Q \Bigr \}.
\end{align*}
The vector space $\tilde{\mathcal Z}_Q$ is the null space of $L_Q$ generated infinitesimally by the family of transformations $\M$ so that $\tilde{\mathcal Z}_Q \subseteq \mathcal Z_Q $ (see \cite{16}).  Note that $\tilde{\mathcal Z}_Q$ is of dimension at most $N'$ but might have strictly smaller dimension if $Q$ has symmetries.  For example, 
\begin{align*}
\tilde{\mathcal Z}_W = \mbox{span} \Bigl \{ \frac{N-2}{2} W + x \cdot \nabla W, \p_{x_j} W, 1 \leq j \leq N \Bigr \}
\end{align*} 
is of dimension $N + 1$.  We say that $Q$ is non--degenerate if $\mathcal Z_Q = \tilde{\mathcal Z}_Q$.  If $Q$ is non--degenerate, 
$\theta \in \M$, then $\theta(Q)$ is also non--degenerate (see \cite{16}).  $W$ is non--degenerate (see \cite{18}).  There are variable sign solutions of 
\eqref{sec24}, the first ones were found by Ding.  More explicit constructions were found by Del Pino, Musso, Pacard, and Pistoia.  
Musso and Wei have shown that these solutions are non--degenerate.  For the precise references see \cite{16}.  A solution 
to \eqref{sec24} satisfies certain pointwise estimates. If $Q \in \Sigma$, then $Q \in C^\infty(\R^N)$ for $N = 3,4$ and 
$Q \in C^4(\R^5)$ if $N = 5$.  Moreover, $\forall \alpha \in \N^N$, $|\alpha|\leq 4$, $\exists C_\alpha > 0$ such that 
\begin{align}\label{sec29}
|\partial_x^\alpha Q(x)| \leq C_\alpha |x|^{-N + 2 - |\alpha|}, \quad |x|\geq 1. 
\end{align}
If $Q \in \Sigma$ and $1 \leq j \leq N$, then we have 
\begin{align}\label{sec210}
\| \partial_{x_j} Q \|^2 = \frac{1}{N} \| \nabla Q \|^2. 
\end{align}
To obtain \eqref{sec210}, we multiply \eqref{sec24} by $x_j \p_{x_j} Q$ and integrate by parts using \eqref{sec29}.  See
\cite{16} for these results.  

If $Q \in \Sigma$, $| \vec \ell | < 1$, we define its Lorentz transform by 
\begin{align}\label{sec211}
Q_{\vec \ell} (t,x) = Q \left (
\left ( -\frac{t}{\sqrt{1 - |\vec \ell|^2}} + \frac{1}{|\vec \ell|^2} \left ( 
\frac{1}{\sqrt{1 - |\vec \ell|^2}} - 1 
\right )
\vec \ell \cdot x
\right ) \vec \ell + x
\right )
= Q_{\vec \ell}(0, x - t\vec \ell). 
\end{align}
$Q_{\vec \ell}$ is a global, non--scattering, type II solution of \eqref{nlw}, traveling in the direction $\vec \ell$. By direct computation, 
we have
\begin{align}
\| \nabla Q_{\vec \ell}(0) \|^2 &= \frac{N - (N-1)|\vec \ell|^2 }{N \sqrt{1 - |\vec \ell|^2}} \| \nabla Q \|^2, \label{sec212} \\
\| \p_t Q_{\vec \ell}(0) \|^2 &= \frac{|\vec \ell|^2 }{N \sqrt{1 - |\vec \ell|^2}} \| \nabla Q \|^2, \nonumber.
\end{align}
Finally, we recall some variational estimates from \cite{24}, \cite{9}, and \cite{10}.  
\begin{lem}\label{lem23}
Assume that $\| \nabla v \|^2 < \| \nabla W \|^2$ and $E(v,0) \leq (1-\delta_0) E(W,0)$ for some $\delta_0 > 0$.  Then 
$\exists \bar \delta = \bar \delta(\delta_0,N)$ such that 
\begin{align}
\| \nabla v \|^2 \leq (1 - \bar \delta) \| \nabla W\|^2, 
\label{sec213} \\
\int \left ( |\nabla v |^2 - |v|^{\frac{2N}{N-2}} \right ) dx \geq \bar \delta \int |\nabla v|^2 dx.  
\label{sec214}
\end{align}
\end{lem}

\begin{lem}\label{lem24}
If 
\begin{align*}
\| \nabla v \|^2 \leq \left ( \frac{N}{N-2} \right )^{\frac{N-2}{2}} \| \nabla W \|^2, 
\end{align*}
then $E(v,0) \geq 0$. 
\end{lem}

\begin{proof}[Sketch of Proof]
Let $C_N$ be as in \eqref{sec25}.  Let
\begin{align*}
f(y) = \frac{1}{2}y - \frac{N-2}{2N} C_N^{\frac{2N}{N-2}} y^{\frac{N}{N-2}}.
\end{align*}
Note that if $\bar y = \| \nabla v \|^2$, then $f(\bar y) \leq E(v,0)$.  Also note that $f(y) = 0 \iff y = 0$ or $y = y^* = 
\left ( \frac{N}{N-2} \right )^{\frac{N-2}{2}} \frac{1}{C_N^N} = \left ( \frac{N}{N-2} \right )^{\frac{N-2}{2}} \| \nabla W \|^2$, 
by \eqref{sec26}.  Thus, Lemma \ref{lem24} follows. 

Note that $f'(y) = 0 \iff y = y_C = \frac{1}{C_N^N} = 
\| \nabla W \|^2$, $f(y_C) = \frac{1}{N C_N} = E(W,0)$ by \eqref{sec26}, and $f''(y_C) \neq 0$.  Thus, $f$ is nonnegative,
strictly increasing in $0 \leq y < y_C$, so that \eqref{sec213} follows. For \eqref{sec214} note that 
\begin{align*}
\int \left ( |\nabla v|^2 - |v|^{\frac{2N}{N-2}} \right ) dx &\geq \int |\nabla v|^2 dx - C_N^{\frac{2N}{N-2}} 
\left ( \int |\nabla v|^2 dx \right )^{\frac{N}{N-2}} \\
&= \left ( \int |\nabla v|^2 dx \right ) \left ( 1 - C_N^{\frac{2N}{N-2}} \left ( 
\int |\nabla v|^2 dx
\right )^{\frac{2}{N-2}}
\right ) \\
&\geq \left ( \int |\nabla v|^2 dx \right ) \left ( 1 - C_N^{\frac{2N}{N-2}} \left ( 
1- \bar \delta 
\right )^{\frac{1}{N-2}} \| \nabla W \|^{\frac{2}{N-2}}
\right ) \\
&= \left ( \int |\nabla v|^2 dx \right ) \left ( 1 -(1 - \bar \delta )^{\frac{2}{N-2}}
\right ).
\end{align*}
This concludes the sketch of the proof.
\end{proof}

\begin{cor}\label{cor25}
There is a constant $c > 0$ such that if for some small $\epsilon > 0$, $\epsilon < \| \nabla v \|^2 \leq 
\left ( \frac{N}{N-2} \right )^{\frac{N-2}{2}} \| \nabla W \|^2 - \epsilon$, then $E(v,0) \geq c \epsilon$.  
\end{cor}

\begin{proof}
The follows from the previous proof and the facts that $f'(0) \neq 0$ and $f'(y^*) \neq 0$. 
\end{proof}

\begin{lem}\label{lem26}
If $\| \nabla v \|^2 \leq \| \nabla W \|^2$ and $E(v,0) \leq E(W,0)$, then 
\begin{align}\label{sec215}
\| \nabla v \|^2 \leq \frac{\| \nabla W \|^2}{E(W,0)} E(v,0) = N E(v,0).
\end{align}
\end{lem}

\begin{proof}
Let $f$ be as before.  Note that $f$ is concave in $\R^+$, $f(0)=0$, $f(\| \nabla W \|^2) = E(W,0)$, and
$f(\| \nabla v \|^2) \leq E(v,0)$. For $s \in (0,1)$, $f(s \|\nabla W \|^2) \geq s f(\|\nabla W\|^2) = sE(W,0)$.
Choose $s = \| \nabla v \|^2 / \| \nabla W \|^2$. 
\end{proof}

These variational estimates will be used in Section 4, 5, and 8

\section{Profile decomposition}

In this section we recall a few facts about the profile decomposition of Bahouri--G\'erard \cite{1} (see also Bulut \cite{3} for extensions
to higher dimensions) as well as some precisions obtained by the authors in \cite{15}.  We also discuss the validity 
(and lack of) of certain Pythagorean expansions (see the addendum to \cite{13}) and an approximation theorem from \cite{1} and \cite{9}, which 
is crucial for applications to nonlinear problems.  
\begin{defn}\label{def31}
Let $\{ (u_{0,n}, u_{1,n} ) \}_n$ be a bounded sequence in $\energ$.  For $j \geq 1$, consider a solution $U^j_L$ of the linear equation \eqref{lw} and a sequence $\{ \lambda_{j,n}, x_{j,n}, t_{j,n} \}_n$ in $(0,\infty) \times \R^N \times \R$.  The sequence
of parameters $\{\lambda_{j,n}, x_{j,n}, t_{j,n} \}_n$, $j\geq 1$, are said to be orthogonal if for all $j \geq 1$
\begin{align}\label{sec31}
j \neq k \implies \lim_{n \rightarrow \infty} \frac{\lambda_{j,n}}{\lambda_{k,n}} + \frac{\lambda_{k,n}}{\lambda_{j,n}} + 
\frac{|t_{j,n} - t_{k,n}|}{\lambda_{j,n}} + \frac{|x_{j,n} - x_{k,n}|}{\lambda_{j,n}} = \infty.
\end{align}
We say that $\left ( U^j_L, \{ \lambda_{j,n}, x_{j,n}, t_{j,n} \}_n \right )_{j \geq 1}$ is a profile decomposition of the sequence 
$\{ (u_{0,n}, u_{1,n} ) \}_n$ if \eqref{sec31} is satisfied and, denoting by 
\begin{align}\label{sec32}
U^j_{L,n}(t,x) = \frac{1}{\lambda_{j,n}^{(N-2)/2}} U^j_L \left ( 
\frac{t - t_{j,n}}{\lambda_{j,n}}, \frac{x - x_{j,n}}{\lambda_{j,n}}
\right )
\end{align}
and 
\begin{align}\label{sec33}
w^J_n(t,x) = S_L(t)(u_{0,n},u_{1,n}) - \sum_{j = 1}^J U^j_{L,n}(t,x), 
\end{align}
the following property holds: 
\begin{align}\label{sec34}
&\limsup_{J \rightarrow \infty} \limsup_{n \rightarrow \infty} \| (w_n^J(0), \partial_t w^J_n(0) )\|_{\energ} < \infty, \nonumber \\
&\lim_{J \rightarrow \infty} \limsup_{n \rightarrow \infty} \| w_n^J \|_{S(\R)} = 0.
\end{align}
\end{defn}

By the paper \cite{1} of Bahouri--G\'erard, if $\{ (u_{0,n}, u_{1,n} ) \}_n$ is a bounded sequence in $\energ$, there exists a subsequence 
(that we will also denote by $\{ (u_{0,n}, u_{1,n} ) \}_n$) that admits a profile decomposition $\left ( U^j_L, \{ \lambda_{j,n}, x_{j,n}, t_{j,n} \}_n \right )_{j \geq 1}$.  \cite{1} establishes this for $N = 3$, but an adaptation of the argument gives the case of general $N$
(see \cite{3}). 

\begin{rmk}\label{rmk31}
The profiles are constructed in \cite{1} as weak limits.  More precisely, they are constructed by, for each $j$, 
\begin{align*}
\vec S_L(t_{j,n}/\lambda_{j,n} )\left ( \lambda_{j,n}^{(N-2)/2} u_{0,n}(\lambda_{j,n}\cdot + x_{j,n}), 
\lambda_{j,n}^{N/2} u_{1,n}(\lambda_{j,n} \cdot + x_{j,n} ) 
\right ) \rightharpoonup_n \vec U^j_L(0), 
\end{align*}
weakly in $\energ$. It is easy to see that, in turn, this property follows from the orthogonality of the parameters 
\eqref{sec31} and the property: 
\begin{align}\label{sec35}
j \leq J \implies \left ( \lambda_{j,n}^{(N-2)/2} w^J_n (t_{j,n} , \lambda_{j,n} \cdot + x_{j,n}), 
\lambda_{j,n}^{N/2} w^J_n (t_{j,n}, \lambda_{j,n} \cdot + x_{j,n} ) 
\right ) \wa_n 0
\end{align}
weakly in $\energ$.
\end{rmk}
 
We next show that \eqref{sec35} holds for any profile decomposition as in Definition \ref{def31}.

\begin{lem}\label{lem32}
For any profile decomposition as in Definition \ref{def31}, \eqref{sec35} holds.  
\end{lem}

In order to prove Lemma \ref{lem32}, we need two further lemmas which will be useful in the sequel.

\begin{lem}\label{lem33}
If $(h_{0,n}, h_{1,n}) \wa_{n} (h_0,h_1)$ in $\energ$, $h_n(t,x) = S_L(t)(h_{0,n},h_{1,n})$, and $h(t,x) = 
S_L(t)(h_0,h_1)$, then 
\begin{align}\label{sec36}
\| h \|_{S(\R)} \leq \limsup_{n \rightarrow \infty} \| h_n \|_{S(\R)}. 
\end{align}
\end{lem}

\begin{proof}
Since $S_L$ is continuous from $\energ$ to $S(\R)$ for the strong topologies, it is also continuous for the weak topologies. As a consequence, $h_n\wa_n h$ weakly in $S(\R)$. The desired conclusion follows from a classical property of weak convergence.
\end{proof}

\begin{lem}\label{lem34}
If $w_n(t,x) = S_L(t)(w_{0,n},w_{1,n})$, $\| (w_{0,n}, w_{1,n} ) \|_{\energ} \leq A$, and $\| w_n \|_{S(\R)} \rightarrow_n 0$, 
then $(w_{0,n}, w_{1,n}) \wa_n (0,0)$ in $\energ$.
\end{lem}

\begin{proof}
Let $(w_{0,n}, w_{1,n}) \wa (h_0,h_1)$ for some subsequence.  By Lemma \ref{lem33}, $h(t) = S_L(t)(h_0,h_1)$ is $0$
in $S(\R)$.  Since $\vec h(t) \in C(\R; \energ)$ it is easy to see that $\vec h$ is 0 in $C(\R ; \energ)$. 
\end{proof}

\begin{proof}[Proof of Lemma \ref{lem32}]
Fix $1 \leq j \leq J$.  Consider, for $J' \geq J$
\begin{align*}
\left (h_{n,0}^{J',j}, h_{1,n}^{J',j} \right ) = \left ( \lambda_{j,n}^{(N-2)/2} w^{J'}_n (t_{j,n} , \lambda_{j,n} \cdot + x_{j,n}), 
\lambda_{j,n}^{N/2} w^{J'}_n (t_{j,n}, \lambda_{j,n} \cdot + x_{j,n} ) 
\right ).
\end{align*}
By \eqref{sec34}, $\left \{ \left ( h^{J',j}_{0,n}, h^{J',n}_{1,n} \right ) \right \}_n$ is bounded in $\energ$, and also by 
\eqref{sec34}
\begin{align*}
\lim_{J' \rightarrow \infty} \limsup_{n \rightarrow \infty} 
\left
\| S_L(\cdot) \left ( h^{J',j}_{0,n}, h^{J',j}_{1,n} \right ) \right \|_{S(\R)} = 0.
\end{align*}
After extraction,
\begin{align*}
 \left (h_{n,0}^{J,j}, h_{1,n}^{J,j} \right ) \wa_n (h_0,h_1)
\end{align*}
weakly in $\energ$. We must prove that $(h_0,h_1)=0$. Arguing by contradiction, we assume that $(h_0,h_1)\neq 0$.
Let $\varepsilon_0=\|S_L(\cdot)(h_0,h_1)\|_{S(\R)}$ and choose $J'>J$ so that 
\begin{align}
 \label{star}
 \limsup_{n\to \infty}\left\| S_L(\cdot)\left (h_{n,0}^{J',j}, h_{1,n}^{J',j} \right )\right\|_{S(\R)}\leq \frac{\eps_0}{2}.
\end{align}
From the 
definitions, we have 
\begin{align*}
w_n^J(t,x) = w^{J'}_n(t,x) + \sum_{k = J+1}^{J'} U^k_{L,n}(t,x).
\end{align*}
Thus, 
\begin{align*}
\left ( h^{J,j}_{0,n}, h^{J,j}_{1,n} \right ) = \left ( h^{J',j}_{0,n}, h^{J',j}_{1,n} \right ) + 
\sum_{k = J+1}^{J'} 
\left ( \lambda_{j,n}^{(N-2)/2} U^k_{L,n} (t_{j,n} , \lambda_{j,n} \cdot + x_{j,n}), 
\lambda_{j,n}^{N/2} \partial_t U^k_{L,n} (t_{j,n}, \lambda_{j,n} \cdot + x_{j,n} ) 
\right ) .
\end{align*}
Recall now that 
\begin{align*}
\Bigl ( &\lambda_{j,n}^{(N-2)/2} U^k_{L,n} (t_{j,n} , \lambda_{j,n} x + x_{j,n}), 
\lambda_{j,n}^{N/2} U^k_{L,n} (t_{j,n}, \lambda_{j,n} x + x_{j,n} ) 
\Bigr) \\ &= 
\left (
\left ( \frac{\lambda_{j,n}}{\lambda_{k,n}} \right )^{(N-2)/2} 
U^k_L \left ( 
\frac{t_{j,n} - t_{k,n}}{\lambda_{k,n}}, 
\frac{\lam_{j,n}x + x_{j,n} - x_{k,n}}{\lambda_{k,n}}
\right ), 
\left ( \frac{\lambda_{j,n}}{\lambda_{k,n}} \right )^{N/2} 
\partial_t U^k_L \left ( 
\frac{t_{j,n} - t_{k,n}}{\lambda_{k,n}}, 
\frac{\lam_{j,n}x + x_{j,n} - x_{k,n}}{\lambda_{k,n}}
\right )
\right ).
\end{align*}
which, since $j \leq J < k \leq J'$, converges weakly to $(0,0)$ in $\energ$, as $n \rightarrow \infty$, by orthogonality of the 
parameters, as can be seen using arguments as in \cite{1}.  
As a consequence,  $\left (h_{n,0}^{J',j}, h_{1,n}^{J',j} \right ) \wa_n (h_0,h_1)$ weakly in $\energ$, which contradicts \eqref{star} in view of Lemma \ref{lem33}.
\end{proof}

We next recall the following Pythagorean expansions established in \cite{1} (see \cite{3} for $N > 3$).  For all $J \geq 1$
\begin{align}
\lim_{n \rightarrow \infty} &
\| u_{0,n} \|_{\dot H^1}^2 + \| u_{1,n} \|^2_{L^2} \nonumber \\ &- 
\left ( \sum_{j = 1}^J \| U^j_{L,n}(0) \|^2_{\dot H^1} +  \| \partial_t U^j_{L,n}(0) \|^2_{L^2}
\right ) - 
\left (  \| w^J_{n}(0) \|^2_{\dot H^1} +  \| \partial_t w^J_{n}(0) \|^2_{L^2} \right ) = 0, \label{sec37}
\end{align}
and 
\begin{align}
\lim_{n \rightarrow \infty} \| u_{0,n} \|_{L^{\frac{2N}{N-2}}}^{\frac{2N}{N-2}} - 
\left ( \sum_{j = 1}^J  \| U^j_{L,n} \|_{L^{\frac{2N}{N-2}}}^{\frac{2N}{N-2}} +  \| w^J_{n} \|_{L^{\frac{2N}{N-2}}}^{\frac{2N}{N-2}}  \right ) = 0. \label{sec38}
\end{align}
In the papers \cite{9}, \cite{10}, \cite{24}, \cite{12}, it is claimed that the following Pythagorean expansions hold: 
\begin{align}
\| u_{0,n} \|_{\dot H^1}^2 &= \sum_{j = 1}^J \| U^j_{L,n}(0) \|^2_{\dot H^1} + 
\| w^J_{0,n} \|^2_{\dot H^1}  + o_n(1), \label{sec39} \\
\| u_{1,n} \|_{L^2}^2 &= \sum_{j = 1}^J \| \partial_t U^j_{L,n}(0) \|^2_{L^2}
+  \| w^J_{1,n} \|^2_{L^2} + o_n(1) \label{sec310}
\end{align}
where $(w_{0,n}^J, w_{1,n}^J) = (w_n^J(0), w_n^J(0))$ and $o_n(1)$ denotes any numerical sequence going to $0$.  Both
properties are false, as explained in \cite{13}.  This leaves a gap in the proofs of the results in \cite{9}, \cite{10}, \cite{24}, and \cite{12}.  This 
gap is filled in the case of \cite{9} and \cite{12} in \cite{13}.  In the present note we will fill the gaps in the proofs in \cite{24} and \cite{10}.   We start 
out by giving a lemma that can, in certain circumstances, be used as a substitute to \eqref{sec39} and \eqref{sec310}.  
Before that, we make the following remark. 

\begin{rmk}\label{rmk37}
By extracting subsequences and possibly changing the profiles by $t$ translates, we can always assume that one of the following holds: $t_{j,n} = 0$ $\forall n$, $-t_{j,n}/ \lambda_{j,n} \rightarrow_n \infty$ or $-t_{j,n} / \lambda_{j,n} \rightarrow -\infty$.   In the 
first case we will say that $j$ is core, and in the other two we will say that $j$ is scattering. 
\end{rmk}

\begin{lem}\label{lem38}
After extraction in $n$, we have that, for any given $\epsilon_0 > 0$, there exists $\bar J = \bar J(\epsilon_0, A)$ 
where $A$ is a uniform bound for the $\energ$ norm of $(u_{0,n},u_{1,n})$, and for each $J \geq \bar J$, 
there exists $\bar n = \bar n(\epsilon_0,A,J)$, such that if $n \geq \bar n$ we have 
\begin{align}
\| u_{0,n} \|_{\dot H^1}^2 &= \sum_{\substack{j = 1 \\ j \in \mathcal J}}^J \| U^j_{L}(0) \|^2_{\dot H^1} + 
\left \| \sum_{\substack{j = 1 \\ j \notin \mathcal J}}^J  U^j_{L,n}(0) \right \|^2_{\dot H^1} +
\| w^J_{n}(0) \|^2_{\dot H^1}  + \epsilon_{n,J}, \label{sec311} \\
\| u_{1,n} \|_{L^2}^2 &= \sum_{\substack{j = 1 \\ j \in \mathcal J}}^J \| \partial_t U^j_{L}(0) \|^2_{L^2} + 
\left \| \sum_{\substack{j = 1 \\ j \notin \mathcal J}}^J \partial_t U^j_{L,n}(0) \right \|^2_{L^2} +
\| \partial_t w^J_{n}(0) \|^2_{L^2}  + \epsilon_{n,J}\label{sec312},
\end{align}
where $|\epsilon_{n,J}| \leq \epsilon_0$, and $\mathcal J$ denotes the set of core indices $j$. 
\end{lem}

We will give the proof of \eqref{sec311}, the one of \eqref{sec312} being similar.  
\begin{proof}
We carry out the proof in a series of steps.  

\emph{Step 1.} Given $\epsilon > 0$, there exists $J_1 = J_1(\epsilon)$, such that, for all $J \geq J_1$ and for all 
$n$, we have 
\begin{align}\label{sec313}
\sum_{j = J_1 + 1}^J \| \vec U^j_L (-t_{j,n}/ \lambda_{j,n}) \|_{\energ}^2 \leq \epsilon.  
\end{align}

To see this, by \eqref{sec37} we have  
\begin{align*}
\| (u_{0,n}, u_{1,n} ) \|_{\energ}^2 = \sum_{j =1}^J \| \vec U_j(0) \|_{\energ}^2 + \| \vec w^J_n(0) \|_{\energ}^2 
+ o_n(1),
\end{align*}
and hence for $n \geq n_1 = n_1(J,A)$ we obtain
\begin{align}\label{sec314}
\sum_{j =1}^J \| \vec U_j(0) \|_{\energ}^2 + \| \vec w^J_n(0) \|_{\energ}^2 \leq 2A.
\end{align}
As a consequence, for each $J$, we have $ \sum_{j = 1}^J \| \vec U^j_L(0) \|_{\energ}^2 \leq 2A$, and thus, for 
appropriate $J_1 = J_1(\epsilon)$
\begin{align*}
\sum_{j = J_1 + 1}^J \| \vec U^j_L(0) \|_{\energ}^2 \leq \epsilon. 
\end{align*} 
By the invariance of the linear energy, Step 1 follows. 

\emph{Step 2.} For each fixed $J$, given $\e > 0$, there exists $n_2 = n_2(\e, J)$ such that, if $\mathcal A 
\subseteq \{ 1, \ldots, J \}$, we have for $n \geq n_2$
\begin{align*}
 \left | \| \vec U^{\mathcal{A}}_{L,n} \|^2_{\energ} - \sum_{j \in \mathcal{A}} \| \vec U^j_L(-t_{j,n}/\lambda_{j,n}) \|_{\energ}^2 
 \right | \leq \e,
\end{align*}
where $\vec U^{\mathcal A}_{L,n} = \sum_{j \in \mathcal A} \vec U^j_{L,n}(0).$  

The proof of Step 2 follows from the fact that 
for $j \neq k$, 
\begin{align*}
\left \la 
\frac{1}{\lambda_{j,n}^{N/2}} \nabla_{t,x} U^j_L \left (\frac{-t_{j,n}}{\lambda_{j,n}},\frac{ x-x_{j,n}}{ \lambda_{j,n}} \right ), 
\frac{1}{\lambda_{k,n}^{N/2}} \nabla_{t,x} U^j_L \left (\frac{-t_{k,n}}{\lambda_{j,n}},\frac{ x-x_{k,n}}{ \lambda_{k,n}} \right )
\right \ra \rightarrow_{n \rightarrow  \infty} 0, 
\end{align*}
by orthogonality of the parameters (here $\la \cdot, \cdot \ra$ denotes the $L^2$ pairing), as in the proof in 
\cite{1} of \eqref{sec37}.  This uses the energy identity 
\begin{align*}
\la \nabla_{t,x} S(t)(f,g), \nabla_{t,x} S(t')(\alpha, \beta) \ra = 
\la (\nabla f, g) , \nabla_{t,x} S(t'-t)(\alpha, \beta) \ra.
\end{align*}

\emph{Step 3.} Let $\epsilon > 0$ be given.  Let $J_1(\epsilon)$ be as in Step 1.  Then $\exists \bar J = \bar J(\epsilon, A)$ such that for each 
$1 \leq j \leq J_1(\epsilon)$, and $J \geq \bar J$, there exists $n_3 = n_3(\e, J)$ so that if $n \geq n_3$, we have
\begin{align}\label{sec315}
\left | \int \nabla U^j_{L,n}(0) \cdot \nabla w_{0,n}^J  dx \right | \leq \e / J_1(\e). 
\end{align}

To see this, recall that 
\begin{align*}
\nabla U^j_{L,n} (0)
= \cos & \left ( \frac{-t_{j,n}}{\lambda_{j,n}} \sqrt{-\Delta} \right ) (\nabla U^j_0) \left ( \frac{x-x_{j,n}}{\lambda_{j,n}} \right )
\frac{1}{\lambda_{j,n}^{N/2}} \\
&+ \frac{\nabla}{\sqrt{-\Delta}}
\sin \left ( \frac{-t_{j,n}}{\lambda_{j,n}} \sqrt{-\Delta} \right ) (U^j_1) \left ( \frac{x-x_{j,n}}{\lambda_{j,n}} \right )
\frac{1}{\lambda_{j,n}^{N/2}}.
\end{align*}
We will deal with the contribution of the first of the two terms to the left hand side of \eqref{sec315}, the second one can 
be treated similarly.  Thus, we have
\begin{align*}
\int
\nabla U^{j}_{0}(x) \cdot \nabla \left ( \cos \left ( \frac{-t_{j,n}}{\lambda_{j,n}} \sqrt{-\Delta} \right ) \left ( \lambda_{j,n}^{(N-2)/2}
w_{0,n}^J(\lambda_{j,n}x + x_{j,n}) \right ) \right ) dx.
\end{align*}
Let $v_n^{J,j} = \cos \left ( (-t_{j,n}/\lambda_{j,n} )\sqrt{-\Delta} \right ) \left ( \lambda_{j,n}^{(N-2)/2}
w_{0,n}^J(\lambda_{j,n}x + x_{j,n}) \right ) $.  For $n \geq n_1(J)$, where $n_1$ is defined in 
\eqref{sec314}, this is a bounded sequence in $\dot H^1$, uniformly in $J$.  After extraction, let $v^{J,j}$ be the weak 
limit in $n$ of $v^{J,j}_n$, in $\dot H^1$.  Since $v^{J,j}$ is bounded in $\dot H^1$, in $J$, for a sequence 
$J_\ell \rightarrow \infty$, let $v^j$ be the weak limit of $v^{J_\ell,j}$ in $\dot H^1$.  We claim that $v^j = 0$.  In fact let 
\begin{align*}
h^{J,j}_n(t,x) = \frac{\lambda_{j,n}^{(N-2)/2}}{2} \left ( w_n^J(t, \lambda_{j,n} x + x_{j,n} ) 
+  w_n^J(-t, \lambda_{j,n} x + x_{j,n} ) \right ). 
\end{align*}
Then, $h^{J,j}_n$ is a solution of the linear wave equation and, for each $j$, $\lim_{J} \limsup_n 
\| h_{n}^{J,j} \|_{S(\R)} = 0$ by \eqref{sec34}.  Since  
$\partial_t h^{J,j}_n(0,x) = 0$, we have 
$h^{J,j}_n(t,x) = \cos (t \sqrt{-\Delta}) \left ( \lambda_{j,n}^{(N-2)/2} w^J_{0,n}(\lambda_{j,n} x + x_{j,n} ) \right )$.  Consider
now $\bar h^{J,n}_n(t,x) = h^{J,j}_n( (t - t_{j,n}) / \lambda_{j,n}, x)$.  We also have $\lim_{J} \limsup_n 
\| \bar h_{n}^{J,j} \|_{S(\R)} = 0$.  Using now Lemma \ref{lem33} and Lemma \ref{lem34} we obtain that $v^j = 0$, and so 
$v^{J,j} \wa_J 0$ in $\dot H^1$.  Let now $\psi \in \dot H^1$, $\epsilon' > 0$ be given.  Choose now $\bar J = \bar J(\epsilon', \psi)$ 
so large so that$
\left | \int \nabla \psi \cdot \nabla v^{J,j} dx \right | \leq \epsilon'/2$, for $J \geq \bar J$. 
Fix $J \geq \bar J$, and choose $\bar n = \bar n(J,\e ',\psi)$ such that $\left | \int \nabla \psi \cdot \nabla v_n^{J,j} dx \right | < \e'/2$, 
for $n \geq \bar n$.  If we choose now $\psi = U^j_0$, $j = 1, \ldots, J_1(\e)$, and $\e' = \e / J_1(\e)$, the proof of Step 3
follows. 

\emph{Step 4.} Let $\e_1$ be given.  Then (after extraction) $\exists \bar J = \bar J(\e_1)$ such that if $J \geq \bar J$, there exists $\bar n = \bar n(\e_1, J)$ such that if $n \geq \bar n$, then 
\begin{align}\label{sec316}
\| u_{0,n} \|_{\dot H^1}^2 = 
\left \| \sum_{j = 1}^J U^j_{L,n} (0) \right \|^2_{\dot H^1} + \| w_{0,n}^J \|_{\dot H^1}^2 + 
\e_{n,J},
\end{align}
where $|\e_{n,J}| \leq \e_1$. 
To obtain Step 4, let $\e_2 > 0$, to be chosen, and set $\e = \e_2 / 10$ in Step 1.  We next use
Step 3 to produce $\bar J = \bar J( \e_2 / 10)$ such that, for $J \geq \bar J$, we can 
find $n_3 = n_3(\e_2 / 10, J)$ so that, if $n \geq n_3$,  
\begin{align}\label{sec317}
\left | 
\sum_{j =1}^{J_1(\e_2 / 10)} \int \nabla U^j_{L,n}(0) \cdot \nabla w^J_{0,n} dx \right | \leq \e_2 / 10.
\end{align} 
Fix $J \geq \bar J$.  Apply Step 1 and Step 2, to conclude that, if $n \geq n_2 ( \e_2 / 10 , J )$, we have 
\begin{align}\label{sec318}
\left \| 
\sum_{k = J_1 (\e_2/10) + 1}^J \frac{1}{\lam_{k,n}^{N/2}} \nabla U^k_L \left ( -\frac{t_{k,n}}{\lam_{k,n}}, \frac{x - x_{k,n}}{\lam_{k,n}} \right )
\right \|^2 \leq \frac{\e_2}{5}.  
\end{align}
Also note that, in light of \eqref{sec314}, if $n \geq n_1(J)$, 
\begin{align}\label{sec319}
\| (w_{0,n}^J, w_{1,n}^J ) \|^2_{\energ} \leq 2 A.
\end{align}
Applying now \eqref{sec314} with $J = J_1(\e_2/10)$ and Step 2, with $\e = A$, $J = J_1(\e_2 / 10)$, we see that for $n \geq 
n_1 ( J_1(\e_2/ 10), A), n_2 (A, \e_2)$, we have 
\begin{align}\label{sec320}
\left \| \sum_{k = 1}^{J_1(\e_2/10)} 
\frac{1}{\lam_{k,n}^{N/2}} \nabla U^k_L \left ( -\frac{t_{k,n}}{\lam_{k,n}}, \frac{x - x_{k,n}}{\lam_{k,n}} \right )
\right \|^2 \leq 3 A.  
\end{align}
Now, 
\begin{align*}
\nabla u_{0,n} &= 
 \sum_{k = 1}^{J_1(\e_2/10)} 
\frac{1}{\lam_{k,n}^{N/2}} \nabla U^k_L \left ( -\frac{t_{k,n}}{\lam_{k,n}}, \frac{x - x_{k,n}}{\lam_{k,n}} \right )
+ 
\sum_{k = J_1(\e_2 / 10)+1}^{J} 
\frac{1}{\lam_{k,n}^{N/2}} \nabla U^k_L \left ( -\frac{t_{k,n}}{\lam_{k,n}}, \frac{x - x_{k,n}}{\lam_{k,n}} \right )
+ \nabla w_{0,n}^J \\
&=: I_n + II_n + III_n. 
\end{align*}
Assume that $n \geq n_2( \e_2 / 10, J), n_1(J), n_1(J_1(\e_2/10), A), n_2(\e_2, A), n_3(\e_2 / 10, J)$. Then 
\begin{align*}
\| \nabla u_{0,n} \|^2 = \| I_n \|^2 + \| II_n \|^2 + \| III_n \|^2 
+ 2 \la I_n , II_n \ra + 2 \la I_n , III_n \ra +  2 \la II_n , III_n \ra.
\end{align*}
By \eqref{sec317} $2 | \la I_n, III_n \ra | \leq \e_2 / 5$.  By \eqref{sec318} and \eqref{sec319}, 
$2 | \la II_n, III_n \ra | \leq 2 (2A)^{1/2} (\e_2 / 5)^{1/2}$.  Moreover, 
$\| I_n \|^2 + \| II_n \|^2 + 2 \la I_n , II_n \ra = \left \| \sum_{j = 1}^J U^j_{L,n} (0) \right \|^2_{\dot H^1}$. Hence choose $\e_2$ so that $\e_2 / 5 + 2 (2A)^{1/2} (e_2/5)^{1/2} \leq \e_1$,  and Step 4 follows. 

\emph{Step 5.} Fix $J$, and suppose $1\leq j \leq J$, $1 \leq j' \leq J$, $j \neq j'$. Let $\e > 0$ be given.  Assume that 
$j$ is core, i.e. $t_{j,n} \equiv 0$.  Then, there exists $n_5 = n_5(\e, J)$, such that 
\begin{align*}
\left | \left \la \frac{1}{\lambda_{j,n}^{N/2}} \nabla U^j_L \left ( 0,
\frac{x - x_{j,n}}{\lam_{j,n}} \right ),
\frac{1}{\lambda_{j',n}^{N/2}} \nabla U^{j'}_L \left ( \frac{-t_{j',n}}{\lam_{j',n}},
\frac{x - x_{j',n}}{\lam_{j',n}} \right )
\right \ra
\right | \leq \e, 
\end{align*}
for $n \geq n_5$.

To see this, if $\frac{\lam_{j,n}}{\lam_{j',n}} + \frac{\lam_{j',n}}{\lam_{j,n}} \rightarrow \infty$, this follows easily using the 
Fourier transform.  If $\lam_{j,n} \simeq \lam_{j',n}$, but $ | t_{j',n} / \lam_{j',n} | \rightarrow \infty$, we approximate
$(U^j_0,U^j_1)$ and $(U^{j'}_0, U^{j'}_1)$ by smooth compactly supported functions. We then do a rescaling to eliminate 
the $\lambda_{j,n}, \lam_{j',n}$. We translate in $x$, to put $x_{j,n} - x_{j',n}$ with $U^{j'}_L$, and then use the 
pointwise bound $ \left |\nabla U^{j'}_L ( -t_{j',n} / \lam_{j',n} ) \right | \leq C \left | t_{j',n} / \lam_{j',n} \right |^{-(N-1)/2}$.
Finally if $\lam_{j,n} \simeq \lam_{j',n}$ and $| t_{j,n} / \lam_{j,n} | \leq C$, we must have that 
$\frac{|x_{j,n} - x_{j',n}|}{\lam_{j,n}} \rar \infty$.   We then rescale the $\lam$'s, put the translations on $U^j_L$, and by the compact support of the approximates to $(U^j_0,U^j_1)$, $(U^{j'}_0,U^{j'}_1)$, we conclude. 

\emph{Step 6.}  In this step, we conclude the proof of \eqref{sec311}.  Because of Step 4, we only need to compare 
\begin{align*}
\left 
\| \sum_{j = 1}^J \frac{1}{\lam_{j,n}^{N/2}} \nabla U^j_L \left ( -\frac{t_{j,n}}{\lam_{j,n}}, \frac{x - x_{j,n}}{\lam_{j,n}} \right )
\right \|^2
\end{align*}
to 
\begin{align*}
\sum_{\substack{j = 1 \\ j \in \mathcal J}}^J \left \| \frac{1}{\lam_{j,n}^{N/2}} \nabla U^j_L \left ( 0, \frac{x - x_{j,n}}{\lam_{j,n}} \right ) \right \|^2 + 
\left \| \sum_{\substack{j = 1 \\ j \notin \mathcal J}}^J \frac{1}{\lam_{j,n}^{N/2}} \nabla U^j_L \left ( -\frac{t_{j,n}}{\lam_{j,n}}, \frac{x - x_{j,n}}{\lam_{j,n}} \right ) \right \|^2.
\end{align*}
This is easily accomplished by using Step 5, and concludes the proof of Lemma 3.8.
\end{proof}

The failure of the expansions \eqref{sec39} and \eqref{sec310}, the positive result Lemma \ref{lem38}, as
well as Remark \ref{rmk37}, naturally bring up the issue of the uniqueness of the profile decomposition, which we now turn to.  
This topic was elucidated in \cite{15}, Lemma 3.2, Lemma 3.3, which we now state without proof.  

\begin{lem}\label{lem39}
Let $\left ( U^j_L, \{ \lam_{j,n}, x_{j,n}, t_{j,n} \}_n \right )_j$ be a  profile decomposition of the sequence $\{ (u_{0,n}, u_{1,n}) \}_n$.  For all $j \geq 1$, consider sequences $\{ \tilde \lam_{j,n}, \tilde x_{j,n}, \tilde t_{j,n} \}_n$ in $(0,\infty) \times \R^N
\times \R$ such that for all $j \geq 1$, there exists  $(\mu_j, y_j, s_j) \in (0,\infty) \times \R^N \times \R$ such that 
\begin{align}\label{sec321}
\lim_{n \rightarrow \infty} \frac{\tilde \lam_{j,n}}{\lam_{j,n}} = \mu_j, \quad 
\lim_{n \rar \infty} \frac{\tilde x_{j,n} - x_{j,n}}{\lam_{j,n}} = y_j, \quad
\lim_{n \rar \infty} \frac{\tilde t_{j,n} - t_{j,n}}{\lam_{j,n}} = s_j.  
\end{align}
Let $\tilde U^j_L(t,x) = \mu_j^{(N-2)/2} U^j_L (s_j + \mu_j t, y_j + \mu_j x)$.  Then 
$\left ( \tilde U^j_L, \{ \tilde \lam_{j,n}, \tilde x_{j,n}, \tilde t_{j,n} \}_n \right )_j$ is also a profile decomposition for the sequence
$\{ (u_{0,n}, u_{1,n} ) \}_n$. 
\end{lem}

 Moreover, we have the following uniqueness result.  

\begin{lem}\label{lem310}
Let $\left ( U^j_L, \{ \lam_{j,n}, x_{j,n}, t_{j,n} \}_n \right )_j$ and $\left ( U^j_L, \{ \tilde \lam_{j,n}, \tilde x_{j,n}, \tilde t_{j,n} \}_n \right )_j$ be two profile decompositions of the same sequence $\{ (u_{0,n}, u_{1,n}) \}_n$.  Assume that each of the sets
\begin{align*}
\mathcal J = \{ j \geq 1 : U^j_L \neq 0 \}, \quad \mathcal K = \{ k \geq 1: \tilde U^k_L \neq 0 \},
\end{align*}
is finite or equal to $\N \backslash \{ 0 \}$. Then, extracting sequences (in $n$) if necessary, there exists a unique one--to--one map 
$\varphi : \N \backslash \{ 0 \} \rar \N \backslash \{ 0 \}$, with the following property.  For all $j\geq 1$, letting $k = \varphi(j)$, 
then $U^j_L = 0$ if and only if $\tilde U^k_L = 0$.  Furthermore, if $U^j_L \neq 0$, there exists $(\mu_j, y_j, s_j) \in 
(0,\infty) \times \R^N \times \R$ such that 
\begin{align*}
\lim_{n \rightarrow \infty} \frac{\tilde \lam_{k,n}}{\lam_{j,n}} = \mu_j, \quad 
\lim_{n \rar \infty} \frac{\tilde x_{k,n} - x_{j,n}}{\lam_{j,n}} = y_j, \quad
\lim_{n \rar \infty} \frac{\tilde t_{k,n} - t_{j,n}}{\lam_{j,n}} = s_j, 
\end{align*}
and $\tilde U^k_L(t,x) = \mu_j^{(N-2)/2} U^j_L (s_j + \mu_j t, y_j + \mu_j x)$.
\end{lem}

\begin{rmk}\label{rmk311}
It is easy to see that if for a sequence $\{ (\lambda_n, x_n, t_n ) \}_n$ in $(0,\infty) \times \R^N \times \R$, 
\begin{align*} \vec S_L ( t_n / \lam_n ) \left  ( \lam_n^{(N-2)/2} u_{0,n}(\lam_n\cdot + x_n) , \lam_n^{N/2}
u_{1,n} ( \lam_n \cdot + x_n ) \right ) \wa (v_0,v_1)\end{align*}
weakly in $\energ$, then modulo a transformation as in Lemma \ref{lem39}, 
$(v_0,v_1) = \vec U^j_L(0)$, for a profile $U^j_L$ in the profile decomposition of $\{ (u_{0,n}, u_{1,n}) \}_n$.  
\end{rmk}

In order to apply the profile decomposition to nonlinear problems, we need to discuss the notion of nonlinear profiles.  

\begin{defn}\label{def311}
Let $j \geq 1$.  A nonlinear profile $U^j$ associated to the linear profile $U^j_L$ and the sequence of parameters 
$\{ \lam_{j,n}, t_{j,n} \}_n$ is a solution $U^j$ of \eqref{nlw} such that, for large $n$, $-t_{j,n} / \lam_{j,n} \in 
I_{\max}(U^j)$ and 
\begin{align*}
\lim_{n \rar \infty} \left \|
\vec U^j_L( -t_{j,n} / \lam_{j,n} ) - \vec U^j ( -t_{j,n} / \lam_{j,n} )
\right \|_{\energ} = 0. 
\end{align*}
\end{defn}

Extracting subsequences, we can always assume that for all $j \geq 1$, the following limit exists: 
\begin{align}\label{sec322}
\lim_{n \rar \infty} -\frac{t_{j,n}}{\lam_{j,n}} = \sigma_j \in [-\infty,\infty].
\end{align}
Using the local theory of the Cauchy problem for \eqref{nlw} if $\sig_j \in \R$, and the existence of wave operators for 
\eqref{nlw} if $\sig_j \in \{-\infty, \infty\}$, we obtain that for all $j$, there exists a unique nonlinear profile $U^j$ associated 
to $U^j_L$ and $\{\lam_{j,n}, t_{j,n} \}_n$. If $\sig_j \in \R$ then $\sig_j \in (T_-(U^j), T_+(U^j) )$.  If $\sig_j = -\infty$ then 
$T_-(U^j) = -\infty$ and $U^j$ scatters backward in time.  Finally, if $\sig_j = \infty$, then $T_+(U^j) = \infty$ and $U^j$ scatters
forward in time. Denoting by 
\begin{align}\label{sec323}
U^j_n(t,x) = \frac{1}{\lam_{j,n}^{(N-2)/2}} U^j \left ( 
\frac{t - t_{j,n}}{\lam_{j,n}}, \frac{x - x_{j,n}}{\lam_{j,n}}
\right ),
\end{align}
we see that the maximal positive time of existence of $U^j_n$ is exactly $\lam_{j,n} T_+(U^j) + t_{j,n}$
(or $\infty$ if $T_+(U^j) = \infty$).  The nonlinear profiles are used to approximate nonlinear solutions.  The main 
result here is: 

\begin{thm}\label{thm312}
Let $\{ (u_{0,n}, u_{1,n} ) \}_n$ be a bounded sequence in $\energ$ which admits a profile decomposition. Let $\theta_n \in 
(0,\infty)$.  Assume that for all $j,n$, 
\begin{align*}
\frac{\theta_n - t_{j,n}}{\lam_{j,n}} < T_+(U^j) \quad \mbox{and} \quad 
\limsup_{n \rar \infty} \| U^j \|_{S\left ( -\frac{t_{j,n}}{\lam_{j,n}}, \frac{\theta_n - t_{j,n}}{\lam_{j,n}} \right )} < \infty. 
\end{align*}
Let $u_n$ be the solution of \eqref{nlw} with initial data $(u_{0,n}, u_{1,n})$.  Then, for large $n$, $u_n$ is defined 
on $[0,\theta_n)$, and $\limsup_n \| u_n \|_{S(0,\theta_n)} < \infty$. Moreover, for $t \in [0,\theta_n)$, we have 
\begin{align}
\vec u_n(t,x) = \sum_{j = 1}^J \vec U^j_n(t,x) + \vec w^J_n(t,x) + \vec r_n^J(t,x), \label{sec324} 
\end{align}
where 
\begin{align}\label{sec325}
\lim_{J \rar \infty} \left [ 
\limsup_{n \rar \infty} \left ( \| r_n^J \|_{S(0,\theta_n)} + 
\sup_{t \in [0,\theta_n)} \| \vec r_n^J \|_{\energ} \right )
\right ] = 0.
\end{align}
\end{thm}
This theorem was proved for $N = 3$ and the defocusing nonlinear wave equation in \cite{1}.   For the case $3 \leq N \leq 5$ and 
the focusing wave equation, this was proved in \cite{9}, using the results in \cite{1}.  The case $N \geq 6$ is in \cite{30}.  The proof uses the 
following Pythagorean expansion of the norms on the profiles, in $S$, which follows from the orthogonality \eqref{sec31} (see \cite{1}): 
\begin{align}\label{sec326}
\forall J \geq 1, \quad 
\lim_{n \rightarrow \infty} \left [ \left \| 
\sum_{j = 1}^J U^j_n
\right \|_{S(0,\theta_n)}^{2(N+1)/(N-2)}
- 
\sum_{j = 1}^J  \left \| U^j_n
\right \|_{S(0,\theta_n)}^{2(N+1)/(N-2)} \right ] = 0. 
\end{align}
An important consequence of this result is the following.

\begin{cor}\label{cor313}
For any sequence $\{ \sig_n \}_n$ such that $0 < \sig_n < \theta_n$, we have, for a fixed $J$, 
\begin{align}\label{sec327}
\lim_{n \rar \infty} \left [
 \| \vec u_n(\sig_n) \|^2_{\energ} - 
\sum_{j = 1}^J \left \| \vec U^j \left ( \frac{\sig_n - t_{j,n}}{\lam_{j,n}} \right ) \right \|^2_{\energ}
- \| \vec w_n^J(\sig_n) \|^2_{\energ}
\right ] = 0.
\end{align}
Also, $\| u_n \|_{S(0,\theta_n)}^{2(N+1)/(N-2)} = \sum_{j =1}^J \| U^j_n \|^{2(N+1)/(N-2)}_{S(0,\theta_n)} + \e_{n,J}$, 
where $\lim_J \limsup_n |\e_{n,J}| = 0$. 
\end{cor}

\begin{proof}
The last statement is an easy consequence of \eqref{sec326}, \eqref{sec324}, and \eqref{sec325}.  To prove 
\eqref{sec327}, we show a profile decomposition for $\{ (u_n(\sig_n), \partial_t u_n(\sig_n) )\}_n$.  Let 
$s_{j,n} = \frac{\sig_n - t_{j,n}}{\lam_{j,n}}$.  Extracting, we can assume that $s_{j,n}$ has a limit $s_j \in [-\infty,\infty]$, for each $j$.  Next we observe that there exists a unique solution $V^j_L$ of the linear wave equation such that 
\begin{align}\label{sec328}
\lim_{n \rar \infty} \left \| 
\vec V^j_L ( s_{j,n} ) - \vec U^j( s_{j,n}) \right \|_{\energ} = 0.
\end{align}
Indeed, if $s_j \in \R$, it follows from 
\begin{align*}
\limsup_{n \rar \infty} \| U^j_n \|_{S(0,\theta_n)} = \limsup_{n \rar \infty} 
\| U^j \|_{S \left ( -\frac{t_{j,n}}{\lam_{j,n}}, \frac{\theta_n - t_{j,n}}{\lam_{j,n}} \right )} < \infty,
\end{align*}
that $s_j \in I_{\max}(U^j)$, and $V^j_L$ is the solution of the linear wave equation with initial data $\vec U^j(s_j)$ at $t = 
s_j$.  If $s_j = \infty$ (respectively $-\infty$), then $U^j$ scatters forward in time (respectively backward in time) and the existence
of $V^j_L$ follows from the existence of the wave operator.   Letting $t'_{j,n} = -\sig_n + t_{j,n}$, we see that 
$\{ \lam_{j,n}, x_{j,n}, t'_{j,n} \}_n$, $j \geq 1$ verifies \eqref{sec31} and that 
$\left ( V^j_L, \{ \lam_{j,n}, x_{j,n}, t'_{j,n} \}_n \right )_j$ is a profile decomposition of $\{ \vec u_n( \sig_n) \}_n$, because 
of \eqref{sec324}, \eqref{sec325}, and \eqref{sec328}.  We now obtain \eqref{sec327} from \eqref{sec37} and \eqref{sec328}.  
\end{proof}

In \cite{15} we introduced a pre--order in the set of profiles, which is very useful for applications of Theorem \ref{thm312}. 

\begin{note}\label{note314}
If $j$ and $k$ are indices, we write 
\begin{align*}
\left ( U^j, \{ t_{j,n}, \lam_{j,n} \}_n \right ) \preceq \left ( U^k, \{ t_{k,n}, \lam_{k,n} \}_n \right ), 
\end{align*}
if one of the following holds: 
\begin{enumerate}[(a)]
\item the nonlinear profile $U^k$ scatters forward in time,
\item the nonlinear profile $U^j$ does not scatter forward in time and 
\begin{align}\label{sec329}
\forall T \in \R, T < T_+(U^j) \implies 
\lim_{n \rightarrow \infty} \frac{\lam_{j,n} T + t_{j,n} - t_{k,n}}{\lam_{k,n}} < T_+(U^k).
\end{align}
\end{enumerate}
\end{note}

If there is no ambiguity in the profile decomposition that we are discussing, we simply write $(j) \preceq (k)$.  We write 
$(j) \simeq (k)$ if $(j) \preceq (k)$ and $(k) \preceq (j)$, and $(j) \prec (k)$ if $(k) \preceq (j)$ does not hold. 
As usual, we extract subsequences so that the limit appearing in \eqref{sec329} exists for all $j$, $k$, and $T < T_+(U^j)$ (see
\cite{15}). Note that if $U^{j_0}$ scatters forward in time, then $(j) \preceq (j_0)$ for all $j \geq 1$.  Note also that 
\begin{align}\label{sec330}
 (j) \preceq (k) \mbox{ and } U^j \mbox{ scatters forward in time}  \implies 
U^k \mbox{ scatters forward in time}.  
\end{align}
If $U^k$ does not scatter forward in time and $U^j$ scatters forward in time, then $(j) \preceq (k)$ does not hold, i.e. $(k) \prec (j)$. 

The relation $(j) \preceq (k)$ is equivalent to the fact that, if for a sequence of times $\{\tau_n\}_n$, the sequence 
$\left \{ \|U^j_n \|_{S(0,\tau_n)} \right \}_n$ is bounded, then the sequence $\left \{ \|U^k_n \|_{S(0,\tau_n)} \right \}_n$
is also bounded.  In \cite{15} it is proved that the binary relation $\prec$ is a total pre--order on the set of indices. 

\begin{defn}\label{defn315}
We say that the profile decomposition is well--ordered if $\forall j \geq 1$, $(j) \preceq (j+1)$.  
\end{defn}

\begin{clm}\label{clm316}
For any sequence $\{ (u_{0,n}, u_{1,n} ) \}_n$ bounded in $\energ$, with a profile decomposition $\left ( 
U^j_L, \{ \lam_{j,n}, x_{j,n}, t_{j,n} \}_n \right )_j$, there exists a subsequence (in $n$) and a one--to--one mapping 
$\varphi$ of $\N \backslash \{ 0 \}$ such that  $\left ( 
U^{\varphi(j)}_L, \{ \lam_{\varphi(j),n}, x_{\varphi(j),n}, t_{\varphi(j),n} \}_n \right )_j$ is a well--ordered profile decomposition
of $\{ (u_{0,n}, u_{1,n} ) \}_n$.  Moreover, transformations as in Lemma \ref{lem39} preserve well--orderness, and, in fact, 
$\forall j \geq 1$ we have  $\left ( 
U^j_L, \{ \lam_{j,n}, t_{j,n} \}_n \right ) \simeq 
\left ( 
\tilde U^j_L, \{ \tilde \lam_{j,n}, \tilde t_{j,n} \}_n \right )$.  
\end{clm}
With this notion, the approximation result, Theorem 
\ref{thm312}, can be reformulated in the following way. 

\begin{ppn}\label{ppn317}
Assume that for all $j$, $U^j$ scatters foward in time. Then, for large $n$, $u_n$ scatters forward in time, and letting 
$
r_n^J(t) = u_n(t) - \sum_{j = 1}^J U^j_n(t) - w^J_n(t)$, $t \geq 0,$
we have
\begin{align*}
\lim_{J \rar \infty} \left [ 
\limsup_{n \rar \infty} \left ( \| r_n^J \|_{S(0,\infty)} + 
\sup_{t \in [0,\infty)} \| \vec r_n^J \|_{\energ} \right )
\right ] = 0.
\end{align*}
\end{ppn}

\begin{ppn}\label{ppn318}
Assume that the profiles $U^j$ are well--ordered.  Let $T < T_+(U^1)$.  Let $\theta_n = \lambda_{1,n} T + t_{1,n}$, and 
assume that $\theta_n > 0$ for large $n$.  Then, for large $n$, $[0,\theta_n] \subset I_{\max}(u_n)$ and for all $j$ and large 
$n$, $[0,\theta_n] \subset I_{\max}(U^j_n)$.  Furthermore, letting $
r_n^J(t) = u_n(t) - \sum_{j = 1}^J U^j_n(t) - w^J_n(t)$, $t \in [0,\theta_n],$ we have 
\begin{align*}
\lim_{J \rar \infty} \left [ 
\limsup_{n \rar \infty} \left ( \| r_n^J \|_{S(0,\theta_n)} + 
\sup_{t \in [0,\theta_n)} \| \vec r_n^J \|_{\energ} \right )
\right ] = 0.
\end{align*}
\end{ppn}

For the proofs of the Propositions, from Theorem \ref{thm312}, see \cite{15}.

\section{Concentration--compactness, I}

In \cite{24}, the main result is the following: 

\begin{thm}\label{thm41}
Let $(u_0,u_1) \in \dot H^1 \times L^2$, $3 \leq N \leq 5$.  Assume that $E(u_0,u_1) < E(W,0)$.  Let $u$ be the 
corresponding solution of the Cauchy problem \eqref{nlw} with maximal interval of existence $I = I_{\max}(u) = 
(T_-(u_0,u_1), T_+(u_0,u_1) )$.  Then 
\begin{enumerate}[(i)]
\item if $\int |\nabla u_0 |^2 dx < \int |\nabla W|^2 dx$, then $I = \R$ and $\| u \|_{S(\R)} < \infty$ ($u$ scatters in both time 
directions), 
\item if $\int |\nabla u_0 |^2 dx > \int |\nabla W|^2 dx$, then $-\infty < T_- < T_+ < \infty$.
\end{enumerate}
There is no such $u_0$ with $\int |\nabla u_0|^2 dx = \int |\nabla W |^2 dx$. 
\end{thm}

The proof of (i) follows the concentration--compactness/rigidity theorem method due to the second and third author, first
introduced in \cite{23} and further developed in \cite{24}, \cite{25}.  The first step is establishing the variational estimates 
\eqref{sec213} and \eqref{sec214} described in the second section.  After
that, one proceeds by contradiction, assuming that (i) fails.  One then constructs a \lq critical element' and proves its \lq compactness'. This is the concentration--compactness step. Finally (and this is the bulk of the paper \cite{24}), one proves a \lq rigidity' theorem, which 
gives the desired contradiction. For the concentration--compactness step, one considers the statement 
\begin{description}
\item[$(\mathcal{SC})$] For all $(u_0,u_1) \in \energ$, with $\int |\nabla u_0|^2 dx < \int |\nabla W|^2 dx$, $E(u_0,u_1) < 
E(W,0)$, the corresponding solution has $I = \R$ and $\| u \|_{S(\R)} < \infty$. 
\end{description}
For a fixed $(u_0,u_1)$ verifying the hypothesis of (i), we say that $\mathcal{SC}(u_0,u_1)$ holds if the conclusion of $(i)$ holds for 
$u$.  By the local theory of the Cauchy problem, $\exists \tilde \de>0$ such that if $\| (u_0,u_1) \|_{\energ} \leq \tilde \de$, then 
$\mathcal{SC}(u_0,u_1)$ holds.  Hence, by the variational estimates in Section 2, there exists $\eta_0 > 0$ such that if $(u_0,u_1)$
is as in (i) and $E(u_0,u_1) \leq \eta_0$, then $\mathcal{SC}(u_0,u_1)$ holds. Moreover, for any $(u_0,u_1)$ as in $(\mathcal{SC})$, 
$E(u_0,u_1) \geq 0$.  Thus there exists $E_C$ with $\eta_0 \leq E_C \leq E(W,0)$ such that, if $(u_0,u_1)$ is as in $(\mathcal{SC})$
and $E(u_0,u_1) < E_C$, then $\mathcal{SC}(u_0,u_1)$ holds, and $E_C$ is optimal with this property.  One then notices that 
(i) is the statement $E_C = E(W,0)$ and assumes, to reach a contradiction, that $E_C < E(W,0)$.  Under this assumption, 
the following propositions were stated in \cite{24}.  
\begin{ppn}\label{ppn42}
There exists $(u_{0,C},u_{1,C})$ in $\energ$ with $E(u_{0,C},u_{1,C}) = E_C < E(W,0)$ and such that, the corresponding solution
$u_C$ has $\| u_C \|_{S(I)} = \infty$.  
\end{ppn}
\begin{ppn}\label{ppn43}
Let $u_C$ is as in Proposition \ref{ppn42}, $0 \in I$, and assume (say) $\| u_C \|_{S(I_+)} = \infty$, where $I_+ = I \cap [0,\infty)$.  
Then, there exist $x(t) \in \R^N$, $\lambda(t) \in \R^+$, $t \in I_+$ such that $K = \{ \vec v(t,x) : t \in I_+ \}$, has the property that 
$\bar K$ is compact in $\energ$, where 
\begin{align*}
\vec v(t,x) = \left ( \frac{1}{\lambda(t)^{(N-2)/2}} u_C \left ( t , \frac{x - x(t)}{\lambda(t)} \right ), 
\frac{1}{\lambda(t)^{N/2}} \partial_t u_C \left ( t , \frac{x - x(t)}{\lambda(t)} \right ) \right ). 
\end{align*}
\end{ppn}

The contradiction is reached in \cite{24} by showing that the only solution $u$ verifying the conclusion of Proposition \ref{ppn43} is $u = 0$, contradicting $E_C \geq \eta_0 > 0$.  The proofs of Proposition \ref{ppn42} and Proposition \ref{ppn43} are not given in 
\cite{24}, but they implicitly assumed the false Pythagorean identity \eqref{sec39} (see Remark 4.4 in \cite{24}). This leaves 
a gap in the proof in \cite{24}, which we are now going to fill. 

Note that, since $E_C < E(W,0)$, $E_C = (1 - \bar \de) E(W,0)$. By the variational arguments in Section 2, there exists $\de_1 = 
\de_1(\de_0) > 0$ such that if $E(u_0,u_1) \leq ( 1 - \de_0 ) \EW$  and $\| \nabla u_0 \|^2 < \| \nW \|^2$, then 
$\| \nabla u_0 \|^2 \leq ( 1 - \delta_1 ) \| \nW \|^2$ and $E(u_0,u_1) \geq \delta_1 ( \| \nabla u_0 \|^2 + \| u_1 \|^2 )$.  Our
first step is the following lemma. 

\begin{lem}\label{lem44}
Assume that $\{ (u_{0,n}, u_{1,n} ) \}_n \in \energ$ is a sequence such that $E(u_{0,n}, u_{1,n}) \leq ( 1 - \delta_0) 
\EW$, $\| \nabla u_{0,n} \|^2 \leq ( 1 - \de_1 ) \| \nW \|^2$.  Let, after extraction, $\left ( U^j_L, \{ \lambda_{j,n}, x_{j,n},
t_{j,n} \}_n \right )_{j}$ be a profile decomposition of $\{ (u_{0,n}, u_{1,n}) \}_n$.  Then there exist $\de_2 > 0$, $\de_3 > 0$, 
$c_0 > 0$, $\bar J > 0$ large, and, for $J \geq \bar J$, $\bar n = \bar n(J)$, such that for all $1 \leq j \leq J$ and $n \geq \bar n$, 
\begin{align*}
E \left ( 
U^j_L( -t_{j,n}/\lambda_{j,n} ), \partial_t U^j_L( -t_{j,n} / \lambda_{j,n} )  
\right ) &\geq c_0 \| (U^j_0,U^j_1) \|_{\energ}^2, \\
E( w_{0,n}^J, w_{1,n}^J) &\geq c_0 \| (w_{0,n}^J, w_{1,n}^J ) \|_{\energ}^2, \\
E \left ( 
U^j_L( -t_{j,n}/\lambda_{j,n} ), \partial_t U^j_L( -t_{j,n} / \lambda_{j,n} )  
\right ) &\leq ( 1 - \delta_2 ) E(W,0), \\
E( w_{0,n}^J, w_{1,n}^J) &\leq (1 - \de_2) \EW, \\
\| \nabla U^j_L( -t_{j,n} / \lambda_{j,n} ) \|^2 &\leq ( 1 - \de_3) \| \nW \|^2, \\
\| \nabla w^J_{0,n} \|^2 &\leq ( 1- \de_3 ) \| \nW \|^2. 
\end{align*} 
\end{lem}

\begin{proof}
Apply Lemma \ref{lem38}, with $\epsilon_0 = \de_1 \| \nW \|^2 /2$.  Let $\bar J$ be as in Lemma \ref{lem38}, and for $J \geq \bar J$, $\bar n_1 = \bar n$ in Lemma \ref{lem38}. By \eqref{sec311} we obtain $\| \nabla w^J_{0,n} \|^2 \leq ( 1 - \de_1/2 )
\| \nW \|^2$, and for each $j$ core, $\| \nabla U^j_L(0) \|^2 \leq ( 1 - \de_1 /2) \| \nW \|^2$. By the variational estimate in Lemma
\ref{lem23}, \eqref{sec214}, 
we deduce that $E(w_{0,n}^J, w^J_{1,n}) \geq 0$, $E( U^j_L(0), \partial_t U^j_L(0) ) \geq 0$, for $n \geq \bar n_1$, and $j$
core.  Consider now $j$ scattering, i.e. such that $| t_{j,n} / \lambda_{j,n} | \rightarrow_n \infty$.  Note that for this $j$, we have
$\| U^j_L(-t_{j,n} / \lambda_{j,n}) \|_{L^{2N/(N-2)}} \rightarrow_n 0$, as can be seen from the dispersive decay estimate
$| U^j_L(t) | \leq C |t|^{-(N-1)/2}$ for $(U_0^j, U^1_j) \in C^\infty_0 \times C^\infty_0$, and an approximation argument.  Hence, 
$\exists \bar n_2 = \bar n_2(\bar J)$, such that for $1 \leq j \leq J$, $j$ scattering and $n \geq \bar n_2$, we have 
\begin{align*}
E \left ( 
U^j_L( -t_{j,n}/\lambda_{j,n} ), \partial_t U^j_L( -t_{j,n} / \lambda_{j,n} )  
\right ) &\geq \frac{2}{5} \left ( \| \nabla U^j_L( -t_{j,n} / \lambda_{j,n} ) \|^2 + \| \partial_t U^j_L(-t_{j,n}/\lambda_{j,n} ) \|^2 
\right ) \\ &\geq \frac{2}{5} \| \nabla U^j_L( -t_{j,n} / \lambda_{j,n} ) \|^2. 
\end{align*} 
We now apply the Pythagorean expansion of the nonlinear energy, which follows by combining \eqref{sec37} and \eqref{sec38}.  
Then 
\begin{align}\label{sec41}
E(u_{0,n}, u_{1,n}) = \sum_{j = 1}^J E \left ( 
U^j_L( -t_{j,n}/\lambda_{j,n} ), \partial_t U^j_L( -t_{j,n} / \lambda_{j,n} )  
\right ) + E  ( w_{0,n}^J, w_{1,n}^J  ) + o_n(1). 
\end{align}
Suppose now that $j$ is core, and $n \geq \bar n_1, \bar n_2$.  We see then that $E(U^j_L(0), \partial_t U^j_L(0))
\leq (1 - \de_0/2) \EW$, for $n$ large, since all energies are nonnegative. The same argument gives $E(w_{0,n}^J, w_{1,n}^J)
\leq ( 1- \de_0/2) \EW$, again for $n$ large.  Using now \eqref{sec41} for $j$ scattering, we see that, for $n$ large, we have 
\begin{align*}
\frac{2}{5} \| \nabla U^j_L(-t_{j,n}/\lambda_{j,n} ) \|^2 \leq ( 1 - \de_0/2) \EW \leq \frac{1}{N} \| \nW \|^2, 
\end{align*}
so that $ \| \nabla U^j_L(-t_{j,n}/\lambda_{j,n} ) \|^2 \leq 5 \| \nW \|^2 / (2N)$ and hence, since $5/(2N) < 1$, for suitable 
$\de_3$ we have, for all $j$, $n$ large
$ \| \nabla U^j_L(-t_{j,n}/\lambda_{j,n} ) \|^2 \leq ( 1 - \de_3 ) \| \nW \|^2$ and $\| \nabla w^J_{0,n} \|^2 
\leq ( 1 -\de_3 ) \| \nW \|^2$. Returning to \eqref{sec41} and using that all the energies are positive, we find, for $n$ large
and suitable $\de_2$, that, for all $j \leq J$, 
$E \left ( 
U^j_L( -t_{j,n}/\lambda_{j,n} ), \partial_t U^j_L( -t_{j,n} / \lambda_{j,n} )  
\right ) \leq ( 1 - \de_2 ) \EW$, $
E(w_{0,n}^J, w_{1,n}^J) \leq (1 - \de_2 ) \EW$.  The lower bounds on the energy now follow from the variational estimate 
\eqref{sec214}.  
\end{proof}

The next ingredient is: 

\begin{lem}\label{lem45}
Let $\{ (u_{0,n}, u_{1,n}) \}_n$ be a sequence in $\energ$ such that $E(u_{0,n}, u_{1,n}) \rightarrow E_C$, 
$\| \nabla u_{0,n} \|^2 < \| \nW \|^2$, with maximal interval of existence $I_n$, $0 \in I_n$, and such that 
\begin{align*}
\| u_n \|_{S(I_{n,+})} = \infty, \quad \lim_{n \rightarrow \infty} \| u_n \|_{S(I_{n,-})} = \infty, 
\end{align*}
where $I_{n,+} = I_n \cap [0,\infty)$, $I_{n,-} = I_n \cap (-\infty,0]$.  Let (after extraction) $\left ( U^j_L, \{\lambda_{j,n}, 
x_{j,n}, t_{j,n} \}_n \right )_j$ be a profile decomposition of $\{ (u_{0,n}, u_{1,n} ) \}_n$.  Then (after reordering the profiles)
$U^j_L \equiv 0$ for all $j \geq 2$, there exists $C_0$ such that $|t_{1,n}/ \lam_{1,n} | \leq C_0$, and 
$\| (w_{0,n}^J, w_{1,n}^J ) \|_{\energ} \rightarrow_n 0$.  
\end{lem}

\begin{proof}
Since $E_C < E(W,0)$, from the hypothesis and the variational estimates in Section 2, we can fix $\de_0$, $\de_1$ positive, so that 
$E(u_{0,n}, u_{1,n}) \leq ( 1 - \de_0 ) \EW$, $\| \nabla u_{0,n} \|^2 \leq ( 1 - \de_1) \| \nW\|^2$, for $n$ large. Thus, Lemma 
\ref{lem44} can be applied.  Assume that $U^1_L$, $U^2_L$ are both non-zero.  Then $\exists \e > 0$ such that 
$\| (U^j_0,U^j_1) \|^2_{\energ} \geq \e$, $ j = 1,2$ and, by preservation of the linear energy, 
$\left \| \left ( U^j_L( -t_{j,n} / \lam_{j,n} ), \partial_t U^j_L( -t_{j,n} / \lam_{j,n} ) \right ) \right \|^2_{\energ} \geq \e$, $j = 1,2$. 
Thus, by Lemma \ref{lem44}, for $n$ large, we have $E \left ( U^j_L( -t_{j,n} / \lam_{j,n} ), \partial_t U^j_L( -t_{j,n} / \lam_{j,n} ) \right ) \geq \tilde \e$, $j = 1,2$.  Using now \eqref{sec41} and the nonnegativity of all the energies, we see that, for any $j$, $n$ 
large, we have
\begin{align*}
E \left ( U^j_L( -t_{j,n} / \lam_{j,n} ), \partial_t U^j_L( -t_{j,n} / \lam_{j,n} ) \right ) \leq E_C - \tilde \e/2, \quad
\left \| \nabla U^j_L( -t_{j,n} / \lam_{j,n} ) \right \|^2 < \| \nW \|^2.  
\end{align*}
Hence, the nonlinear profile $U^j$ exists for all time and scatters, for $j = 1,2$, by the optimality of $E_C$.  Clearly this holds 
now for all $j$.  But this contradicts $\| u_n \|_{S(I_{n,+})} = \infty$ by Proposition 
\ref{ppn317}.  The same argument shows that $\lim_n \| ( w^J_{0,n}, w^J_{1,n}) \|_{\energ} = 0$. To show that 
$| t_{1,n} / \lam_{1,n} | \leq C_0$, since $\| u_n \|_{S(I_{n,+})} = \infty$, we must have $-t_{1,n} / \lam_{1,n} \leq C_1$, again 
by Proposition \ref{ppn317}.  Assume that $-t_{1,n} / \lam_{1,n} \rightarrow -\infty$.  Then for large $n$, 
$\left \| \frac{1}{\lam_{1,n}^{(N-2)/2}} U^1_L \left ( \frac{t - t_{1,n}}{\lam_{1,n}}, \frac{x - x_{1,n}}{\lam_{1,n}} \right )
\right \|_{S(-\infty,0)} \leq \frac{\tilde \de}{2}$, where $\tilde \de$ is as in the local Cauchy theory and hence, the 
local Cauchy theory gives $\| u_n \|_{S(I_{n,-})} \leq C\de$, for $n$ large, which contradicts $\lim_n \| u_n \|_{S(I_{n,-})} = \infty$. 
\end{proof}

\begin{proof}[Proof of Proposition \ref{ppn42}] By the definition of $E_C$, we can choose (possibly using time translation) a sequence $\{ (u_{0,n}, u_{1,n}) \}_n$ as in Lemma \ref{lem45}.  Let $U^1_L$ be the non-zero profile and $U^1$ the associated
nonlinear profile.  If $\| U^1 \|_{S(I_{1}^{+})} < \infty$, then Proposition \ref{ppn317} gives that $\| u_n \|_{S(I_{n,+})} < \infty$, a contradiction.  Here $I_1^+ = I_1 \cap [0,\infty)$, where $I_1$ is the maximal interval of existence of $U^1$.  By \eqref{sec41}, we have $E \left ( U^1_L(-t_{1,n}  / \lam_{1,n} ), \partial_t U^1_L(-t_{1,n}  / \lam_{1,n} ) \right )
\rightarrow_n E_C$ since $E(w^J_{0,n}, w^J_{1,n}) \rightarrow_n 0$.  But then, by the invariance of the nonlinear energy, $E(U^1,\partial_t U^1) = E_C$.  Since for $n$ large, $\left \| \nabla U^1_L( -t_{1,n} / \lam_{1,n} \right ) \|^2 \leq ( 1 - \de_3) 
\| \nW \|^2$, for $n$ large, $\left \| \nabla U^1 ( -t_{1,n}/\lam_{1,n} ) \right \|^2 \leq ( 1 - \de_3/2) \| \nW \|^2$.  Fix $\bar t \in I_1$.  By the variational estimate \eqref{sec213} we have $\| \nabla U^1 (\bar t) \|^2 \leq ( 1 - \tilde \de_3 ) \| \nW \|^2$, and we set 
$u_C = U^1$.   
\end{proof}

\begin{proof}[Proof of Proposition \ref{ppn43}]
Because of the continuity of the flow, it suffices to show that if $t_n \rightarrow T_+(u_C)$, after extraction we can find 
$\lam_n$, $x_n$, such that $\left ( \frac{1}{\lam_n^{(N-2)/2}} u_C \left ( t_n , \frac{x - x_n}{\lam_n} \right ),
\frac{1}{\lam_n^{N/2}} \partial_t u_C \left ( t_n , \frac{x - x_n}{\lam_n} \right ) \right )$ converges in $\energ$.  
Consider $\{ (u_C(t_n), \partial_t u_C(t_n) ) \}_n$, which is a sequence verifying the hypothesis of Lemma \ref{lem45}.  Because of the conclusion of Lemma \ref{lem45}, we can assume, after extraction and changing $U^1_L$, that $t_{1,n} = 0$. We thus have 
\begin{align*}
( u_C(t_n) , \partial_t u_C(t_n) ) = \left ( \frac{1}{\lam_{1,n}^{(N-2)/2}} U^1_L\left ( 0, \frac{x - x_{1,n}}{\lam_{1,n}} \right ), 
\frac{1}{\lam_{1,n}^{N/2}} \partial_t U^1_L\left ( 0, \frac{x - x_{1,n}}{\lam_{1,n}} \right ) \right ) + 
(w^J_{0,n}, w^J_{1,n} ),
\end{align*}
where $\lim_{n} \| (w_{0,n}^J, w_{1,n}^J ) \|_{\energ} = 0$
which clearly gives the result. 
\end{proof}

We now state an apparently different version of Theorem \ref{thm41}, which turns out to be equivalent to it.  

\begin{thm}\label{thm46}
Let $(u_0,u_1) \in \dot H^1 \times L^2$, $3 \leq N \leq 5$.  Assume that $E(u_0,u_1) < E(W,0)$.  Then 
\begin{enumerate}[(i)]
\item if $\int |\nabla u_0 |^2 dx  + \int |u_1|^2 dx < \int |\nabla W|^2 dx$, then $I = \R$ and $\| u \|_{S(\R)} < \infty$ ($u$ scatters in both time 
directions), 
\item if $\int |\nabla u_0 |^2 dx + \int |u_1|^2 dx > \int |\nabla W|^2 dx$, then $-\infty < T_- < T_+ < \infty$.
\end{enumerate}
There is no such $(u_0,u_1)$ with $\int |\nabla u_0|^2 dx + \int |u_1|^2 dx = \int |\nabla W |^2 dx$. 
\end{thm}

We will see that Theorem \ref{thm46} and Theorem \ref{thm41} are equivalent.  Our main tool is the following simple 
variational fact: 
\begin{align}\label{sec42}
\mbox{If } \| \nabla u_0 \|^2 \leq \| \nW \|^2, \mbox{ then } \| u_0 \|^{2N/ (N-2)}_{L^{2N/(N-2)}} \leq \| \nabla u_0 \|^2. 
\end{align}
To see this, let $C_N$ be the best constant in the Sobolev embedding $\| u_0 \|_{L^{2N/(N-2)}} \leq C_N 
\| \nabla u_0 \|$.  Recall from Section 2, that $\| \nW \|^2 = C_N^{-N}$, $\EW = N^{-1} \| \nW \|^2$.  Then, as in 
Section 2,  
\begin{align*}
\int |\nabla u_0|^2 dx - \int |u_0|^{2N/ (N-2)} dx 
&\geq \int |\nabla u_0|^2 dx - C_N^{2N/(N-2)} \left ( \int |\nabla u_0|^2 dx \right )^{N/(N-2)} \\
&= \left ( \int |\nabla u_0|^2 dx \right ) 
\left ( 1 - C_N^{2N/(N-2)} \left ( \int |\nabla u_0|^2 dx \right )^{2/(N-2)} \right ) \\
&= 
\left ( \int |\nabla u_0|^2 dx \right ) \left ( 1 - \left ( \int |\nW|^2 dx \right )^{-2/(N-2)} \left ( \int |\nabla u_0|^2 dx \right )^{2/(N-2)} \right )
\end{align*} 
and our claim follows. 

\begin{clm}\label{clm47}
If $E(u_0,u_1) < \EW$, then $\| \nabla u_0 \|^2 < \| \nW \|^2$ if and only if $\| \nabla u_0 \|^2 + \|u_1 \|^2 < \| \nabla W \|^2$. 
\end{clm}

To see this, note that 
\begin{align*}
\| \nabla u_0 \|^2 + \| u_1 \|^2 &< \frac{N-2}{N} \| u_0 \|^{2N/(N-2)}_{L^{2N/(N-2)}} 
+ 2 \EW \\
&\leq \left ( 1 - \frac{2}{N} \right ) \| \nabla u_0 \|^2 + \frac{2}{N} \| \nW \|^2 \\
&\leq \| \nabla W \|^2, 
\end{align*}
where we assume $E(u_0,u_1) < \EW$ and $\| \nabla u_0 \| < \| \nabla W \|$, and we use \eqref{sec42}. 

\begin{clm}\label{clm48}
If $E(u_0,u_1) < \EW$, then $\| \nabla u_0 \|^2 > \| \nabla W \|^2$ if and only if $\| \nabla u_0 \|^2 + \| u_1 \|^2 > \| \nabla W \|^2$. 
\end{clm}

To see this, assume by contradiction that $E(u_0,u_1) < \EW$, $\| \nabla u_0 \|^2 + \| u_1 \|^2 > \| \nabla W \|^2$ and $\| \nabla u_0 \|^2 \leq \| \nabla W \|^2$. By Lemma \ref{lem26}, we must have  $\| \nabla u_0 \|^2 <\| \nabla W \|^2$, which contradicts Claim \ref{clm47}.

Note also that Lemma \ref{lem26} and Claims \ref{clm47} and \ref{clm48} show that $E(u_0,u_1) < E(W,0)$, $\| \nabla u_0 \|^2 + \|u_1 \|^2 =
\| \nabla W \|^2$ is impossible, and thus Theorem \ref{thm46} and Theorem \ref{thm41} are seen to be equivalent.  

\begin{rmk}\label{rmk49}
One can prove Theorem \ref{thm46} (i) using the concentration--compactness/rigidity theorem method and the classical 
Pythagorean expansions \eqref{sec37} and \eqref{sec38}, which, because of the above equivalences gives another approach to the proof of Theorem \ref{thm41} (i) that does not use Lemma \ref{lem38}. 
\end{rmk}

\section{Compact solutions}

As we saw in Section 4, Proposition \ref{ppn43}, \lq compact' solutions are very important in the \lq concentration--compactness/rigidity
theorem' method. The contradiction leading to the proof of Theorem \ref{thm41} (i), was obtained in \cite{24}, through the proof of the 
following rigidity theorem.  

\begin{thm}\label{thm51} 
Let $u$ be a solution to \eqref{nlw} with maximal interval $I$, $0 \in I$.  Suppose that $E(u_0,u_1) < \EW$, $\| \nabla u_0 \|
< \| \nW \|$ and that, for $t \in I_+ = I \cap [0,\infty)$, there exist $\lam(t) \in \R^+$, $x(t) \in \R^N$ such that 
\begin{align*}
K = \left \{ 
\left ( 
\frac{1}{\lam(t)^{(N-2)/2}} u \left ( t, \frac{x - x(t)}{\lam(t)} \right ), 
\frac{1}{\lam(t)^{N/2}} \partial_t u \left ( t, \frac{x - x(t)}{\lam(t)} \right )
\right ) : t \in I_+
\right \}
\end{align*}
has compact closure in $\energ$.  Then $u = 0$. 
\end{thm}

In \cite{24}, this theorem is reduced to the case when, in addition, $\lam(t) \geq A > 0$, $t \in I_+$, and $\int \nabla u_0 u_1 dx = 0$
(this is the momentum of the solution, which is another invariant of the flow).  A crucial part of the argument in \cite{24}
is the fact that compact self--similar blow--up, i.e., the case when $T_+ < \infty$ and $\lam(t) \simeq (T_+ - t)^{-1}$, is 
impossible, which is a fundamental feature of the energy critical problem.  We refer to \cite{24} for the proofs. 

\begin{defn}\label{defn52} 
We say that a solution $u$ to \eqref{nlw}, with maximal interval of existence $I$, is compact if
there exist $\lam(t) \in \R^+$, $x(t) \in \R^N$, $t \in I$ such that 
\begin{align}\label{sec51} 
K = \left \{ 
\left ( 
\frac{1}{\lam(t)^{(N-2)/2}} u \left ( t, \frac{x - x(t)}{\lam(t)} \right ), 
\frac{1}{\lam(t)^{N/2}} \partial_t u \left ( t, \frac{x - x(t)}{\lam(t)} \right )
\right ) : t \in I
\right \}
\end{align}
has compact closure in $\energ$.
\end{defn}

Besides the crucial role in the \lq concentration--compactness/rigidity theorem' method, compact solutions are also fundamental in the study of type II solutions of \eqref{nlw}, that is solutions such that
\begin{align*}
\sup_{0 < t < T_+(u)} \| \vec u(t) \|_{\energ} < \infty.
\end{align*}
In fact, as we will discuss in the next section for any such solution $u$, which does not scatter, there exists $\{ t_n \}_n$, 
$t_n \rar T_+(u)$, such that, up to modulation, $( u(t_n), \partial_t u(t_n) ) \wa_n (U(0),\partial_t U(0))$, weakly in $\energ$,
where $U$ is a compact solution of \eqref{nlw} (see \cite{15}).  Thus, the issue of classifying compact solutions is 
very important.  This leads one to finding extensions of Theorem \ref{thm51}, without size restrictions.  

\begin{rmk}\label{rmk53}
Any solution $Q$ to the elliptic equation $\Delta Q + |Q|^{4/(N-2)} Q = 0$ in $\R^N$, $Q \in \dot H^1$ is a compact solution 
of \eqref{nlw}.  Recall that $W$ is the only (up to sign and scaling) radial solution, and the only nonnegative solution (
up to translation and scaling).  Recall also from Section 2 that variable sign solutions do exist.  As before, we will denote by $\Sigma$ 
the set of non--zero solutions of the elliptic equation.  If $\vec \ell \in \R^N$, $|\vec \ell| < 1$, and $Q \in \Sigma$, let 
\begin{align}\label{sec52} 
Q_{\vec \ell} (t,x) = Q \left (
\left ( -\frac{t}{\sqrt{1 - |\vec \ell|^2}} + \frac{1}{|\vec \ell|^2} \left ( 
\frac{1}{\sqrt{1 - |\vec \ell|^2}} - 1 
\right )
\vec \ell \cdot x
\right ) \vec \ell + x
\right )
= Q_{\vec \ell}(0, x - t\vec \ell),
\end{align}
be the Lorentz transform of $Q$.  
\end{rmk} 

Then, for each $\vl$, $Q$, $Q_{\vl}$ is a traveling wave solution of \eqref{nlw}, and hence a compact solution.  The classification of 
all traveling wave solutions is also an important problem, for instance in light of the soliton resolution conjecture.  We thus have the following two
important conjectures. 

\begin{conj}[Rigidity conjecture for compact solutions]\label{conj54}
$Q_{\vl}$, $Q \in \Sigma$, $|\vl| < 1$, and $0$ are the only compact solutions of \eqref{nlw}.  If the conjecture holds, the
$Q_{\vl}$ are also the only traveling wave solutions. 
\end{conj}

\begin{conj}[Soliton resolution conjecture]\label{conj55}
If $u$ is a type II solution of \eqref{nlw}, there exists $J \in \N$, $Q_j \in \Sigma$, $j = 1,\ldots,J$, $\vl_j \in \R^N$, $|\vl_j | < 1$,
and a solution $v_L$ to the linear wave equation \eqref{lw} (the radiation term) such that, if $t_n \rightarrow T_+(u)$, there exist $\{ \lam_{j,n} \}_n$ a sequence in $\R^+$, $\{ x_{j,n} \}_n$, a sequence in 
$\R^N$, which verify the orthogonality condition $\frac{\lam_{j,n}}{\lam_{k,n}} + \frac{\lam_{k,n}}{\lam_{j,n}}
+ \frac{|x_{j,n} - x_{k,n}|}{\lam_{j,n}} \rar_n \infty$ for $j \neq k$ such that 
\begin{align*}
u(t_n, x) &= \sum_{j = 1}^J \frac{1}{\lam_{j,n}^{(N-2)/2}} Q^j_{\vl_j} \left ( 0, \frac{x-x_{j,n}}{\lam_{j,n}} \right )
+ v_L(t_n,x) + o_n(1) \quad \mbox{in } \dot H^1, \\
\partial_t u(t_n, x) &= \sum_{j = 1}^J \frac{1}{\lam_{j,n}^{N/2}} \partial_t Q^j_{\vl_j} \left ( 0, \frac{x-x_{j,n}}{\lam_{j,n}} \right )
+ \partial_t v_L(t_n,x) + o_n(1) \quad \mbox{in } L^2. 
\end{align*} 
\end{conj}
In connection with Conjecture \ref{conj54}, we have, in the radial case: 

\begin{thm}[\cite{9}]\label{thm55}
If $u$ is a radial compact solution of \eqref{nlw}, then 
$u(t,x) = \pm \frac{1}{\lam_0^{(N-2)/2}} W \left ( \frac{x}{\lam_0} \right )$, for some $\lam_0 > 0$. 
\end{thm}

The proof uses modulation theory, following \cite{18}, virial identities, and the characterization of radial $Q \in \Sigma$ as $W$. 


\begin{rmk}\label{rmk56}
Conjecture \ref{conj55} was proved in the radial case by the authors, when $N = 3$, in \cite{12} for a specific sequence of times, and 
in \cite{14} for all sequences of times.  The result for a specific sequence of times in the radial case was extended to all odd dimensions in 
\cite{30} and to all even dimensions in \cite{6} and \cite{21}.  In the non--radial case, the only progress so far is a perturbative result, which will be 
discussed in Section 8. 
\end{rmk}

We now turn to a discussion of the current state of the art on Conjecture \ref{conj54}, in the non--radial case.   The first result is: 

\begin{thm}\label{thm57}
Assume that $3 \leq N \leq 5$ and $u$ is a compact solution that is not identically 0.  Assume that 
\begin{align}\label{sec53} 
\sup_{t \in I_{\max}} \int |\nabla u(t)|^2 dx < \frac{4 \sqrt{N-1}}{N} \int |\nW |^2 dx.
\end{align} 
Then, $I_{\max} = \R$, and there exist $\ell \in (-1,1)$, a rotation $\mathcal{R} \in \R^N$, $\lam_0 > 0$, $x_0 \in \R^N$, 
$\iota_0 \in \{ \pm 1\}$ such that 
\begin{align}\label{sec54} 
u(t,x) = \frac{\iota_0}{\lam_0^{(N-2)/2}} W_{\ell} \left ( \frac{t}{\lam_0} , \frac{\mathcal R(x) - x_0}{\lam_0} \right ),
\end{align}
where $W_\ell(t,x) = W \left ( \frac{x_1 - t\ell}{\sqrt{1 - \ell^2}} , x' \right ) = W_{\ell \vec e_1} (t,x)$,
$x' = (x_2, \ldots, x_N )$, and $W_{\ell \vec e_1}$ is defined in \eqref{sec52} with $Q = W$, $\vec e_1 = (1,0, \ldots 0)$. 
\end{thm}

\begin{rmk}\label{rmk58}
Theorem \ref{thm57} was first proved in \cite{10} with $4 \sqrt{N-1} / N$ replaced by 2 (note that for $3 \leq N \leq 5$, $
4 \sqrt{N-1}/N > 1$).  Professor Miao kindly pointed out a calculational error in that proof.  Theorem \ref{thm57}
was then proved in the arxiv version of \cite{10}, \cite{11}.  A different proof, which will be given later, was provided 
in the authors' work \cite{16}.  
\end{rmk}

\begin{rmk}\label{rmk59}
As will be proved in Section 6, if $E(u_0,u_1) < \EW$, $\| \nabla u_0 \| < \| \nabla W \|$, then
\begin{align*}
\sup_{t \in I_{\max}} \| \nabla u(t) \|^2 + \frac{N}{2} \| \partial_t u (t) \|^2 < \| \nabla W \|^2. 
\end{align*}
Thus, if $u$ is non--zero and compact, by Theorem \ref{thm57}, we would have $u(t,x) = \frac{\iota_0}{\lam_0^{(N-2)/2}} W_{\ell} \left ( \frac{t}{\lam_0} , \frac{\mathcal R(x) - x_0}{\lam_0} \right )$.  But a calculation (see Section 6) shows that 
\begin{align*}
E( W_\ell(0), \partial_t W_\ell(0)  ) = \frac{1}{\sqrt{1 - \ell^2}} \EW > \EW,
\end{align*}
giving a contradiction, which shows that $u$ must be 0.  Thus, Theorem \ref{thm57} implies Theorem \ref{thm51} for $3 \leq N \leq 5$. 
\end{rmk}

The proof of Theorem \ref{thm57} is an extension to the non--radial case of the proof of Theorem \ref{thm55}, using the fact that 
if $Q \in \Sigma$ ($Q \neq 0$, $\Delta Q + |Q|^{4/(N-2)}Q = 0$ ) and $\int |\nabla Q|^2 dx < 2 \int |\nabla W |^2$, then 
$Q(x) = \frac{\iota_0}{\lam_0^{(N-2)/2}} W \left ( \frac{x - x_0}{\lam_0} \right )$, as we saw in Theorem \ref{thm21}.  Moving forward, an 
important fact about compact solutions, is their invariance under Lorentz transformations.  


\begin{thm}\label{thm58}
Let $u$ be a compact solution of \eqref{nlw}, and $\vl \in \R^N$, $|\vl | < 1$.  After a rotation $\mathcal{R} \in \R^N$, assume that 
$\vl = \ell \vec e_1$, where $|\ell| < 1$.  Then $u_\ell (t,x) = u \left ( \frac{t - \ell x_1}{\sqrt{1 - \ell^2}} , \frac{x_1 - \ell t}{\sqrt{1 - \ell^2}} , x' \right )$ is also a compact solution of \eqref{nlw}.
\end{thm}

\begin{rmk} 
Theorem \ref{thm58} was proved by the authors in \cite{16}. For the definition given of $u_\ell(t,x)$ in Theorem \ref{thm58} to 
make sense, one needs that $u$ is globally defined in $t$.  In \cite{16} a definition of $u_\ell$ was also given in the case of finite time blow--up.  However, subsequently in \cite{17}, the authors have shown that compact solutions of \eqref{nlw} are always globally defined, thus removing this distinction.
The proof of Theorem \ref{thm58} given in \cite{16} is by a direct analysis, using a simple harmonic analysis lemma from \cite{24}.  
\end{rmk}

Our next result, which combines results in \cite{15} and \cite{17} gives the rigidity conjecture for compact solutions (Conjecture \ref{conj54})
in an asymptotic sense. 

\begin{thm}\label{thm510} 
Let $u$ be a non--zero compact solution of \eqref{nlw}.  Then 
\begin{enumerate}[(a)]
\item $T_-(u) = -\infty$ and $T_+(u) = \infty$, 
\item $E(u_0,u_1) > 0$, 
\item there exists two sequences $\{t_n^\pm\}_n$, and two elements $Q^\pm$ of $\Sigma$, and a vector $\vl$, with 
$|\vl|< 1$, such that $\lim_n t_n^\pm = \pm \infty$, and sequences $\{ \lam_n^\pm \}_n$, $\{ x_n^\pm \}_n$ in 
$\R^+$ and $\R^N$ such that 
\begin{align}
\lim_{n \rightarrow \infty} \Bigl [ 
\Bigl \| &
(\lam_n^\pm)^{(N-2)/2} u(t_n^\pm, \lam_n^\pm \cdot + x_n^\pm ) - Q^\pm_{\vl} (0) \Bigr \|^2_{\dot H^1} \nonumber \\ &+ 
\Bigl \| 
(\lam_n^\pm)^{N/2} \partial_t u(t_n^\pm, \lam_n^\pm \cdot + x_n^\pm ) - \partial_t Q^\pm_{\vl} (0) \Bigr \|^2_{L^2} 
\Bigr ]= 0. \label{sec55}
\end{align}
\end{enumerate}
Moreover, $\vl = - P(u_0,u_1) / E(u_0,u_1)$, where $P(u_0,u_1) = \int \nabla u_0 u_1 dx$ is the momentum of $u$.  
\end{thm}

The proof of this result combines virial arguments, the lack of self--similar compact blow--up (\cite{24}), and extensions of the proof of the 
lack of self--similar blow--up to obtain that both $T_+$ and $T_-$ are infinite (\cite{17}).  


\begin{rmk}\label{rmk511}
Using the proof of Theorem \ref{thm58} and Theorem \ref{thm510} one can show that if $u$ is a compact solution, we can find
$\vl$, $|\vl| < 1$ such that $P( u_{\vl}(0), \partial_t u_{\vl}(0) ) = 0$, i.e. the momentum of $u_{\vl}$ is 0. 
\end{rmk}

Our final rigidity result is the most general one known today.  It gives the rigidity conjecture for compact solutions, under a non--degeneracy assumption. 


\begin{thm}\label{thm511}
Let $u$ be a non--zero compact solution of \eqref{nlw}.  Assume that $Q^+$ or $Q^-$ given by Theorem \ref{thm510} satisfies
the non--degeneracy condition.  Then, there exists $Q \in \Sigma$ such that $u = Q_{\vl}$, where 
$\vl = - P(u_0,u_1) / E(u_0,u_1)$ satisfies $|\vl | < 1$.  
\end{thm}

The non--degeneracy condition is discussed in Section 2, after \eqref{sec28}.  Loosely speaking, $Q \in \Sigma$ is non--degenerate if the kernel 
$\mathcal Z_Q$ of the linearized operator $L_Q = -\Delta - \frac{N+2}{N-2} |Q |^{4/(N-2)}$, which is finite dimensional, is 
spanned by the family of transformations leaving the nonlinear equation invariant, i.e. if 
\begin{align*}
\tilde{\mathcal Z}_Q = \mbox{span} \Bigl \{ (2-N)x_j Q + |x|^2  &  \partial_{x_j} Q - 2 x_j x \cdot \nabla Q, \partial_{x_j} Q, 1 \leq j
\leq N,  \\  &  x_j \partial_{x_k} Q - x_k \partial_{x_j} Q, 1 \leq j < k \leq N, \frac{N-2}{2} Q + x \cdot \nabla Q \Bigr \},
\end{align*}
we have $\mathcal Z_Q = \tilde{\mathcal Z}_Q$. It is known that $W$ is non--degenerate \cite{REY}, \cite{18} and that the non--radial solutions constructed by del Pino, Musso, Pacard, and Pistoia are non--degenerate by work of Musso--Wei. 
  There is no known example of $Q \in \Sigma$ which does not satisfy the 
non--degeneracy condition.  

The first step of the proof is to use a Lorentz transformation to reduce to the case of zero momentum, by Remark \ref{rmk511}.
We then use Theorem \ref{thm510}, and obtain a sequence of times $t_n \rightarrow \infty$, such that 
$\{( u(t_n) , \partial_t u(t_n) ) \}_n$ converges to $(Q,0)$, after modulation, where $Q$ is non--degenerate.  We argue by 
contradiction and show that if $u$ is not $Q$, there exists a compact solution $w$ which has the energy of a $S \in \Sigma$, is close to 
$S$ for positive time and is not $S$, where $S$ in addition is non--degenerate.  We then use modulation theory, in the spirit of \cite{18}, to show that $w(t) = S + e^{-\om t}Y + O \left ( e^{-\om^+ t} \right )$ as $t \rar \infty$, where $Y$ is an eigenfunction of the linearized operator near $S$, $-\om^2$ is the corresponding negative eigenvalue and $\om_+ > \om > 0$. 
It is at this stage of the proof that the non--degeneracy assumption is used.  The contradiction is reached by proving that there is no
compact solution $w$ with such an expansion, and this is the core of the proof.  The idea is to use an exterior energy argument 
to rule this out.  Unlike our previous works \cite{9}, \cite{10}, \cite{12}, \cite{14} using a similar method (see also Section 8), the space dimension is 
not restricted to odd dimensions.  It is based on exterior energy estimates for the linearized equation $\p_t^2 + L_S$, instead of the 
free wave equation and depends on a quantitative unique continuation result of Meshkov \cite{29} for eigenfunctions of elliptic 
operators, which is crucial for us.  

\begin{rmk}\label{rmk512}
Under some size restrictions on the compact solution $u$, we can conclude that $Q = W$ (see \cite{16}).  
In fact if 
$\vl = - P(u_0,u_1) / E(u_0,u_1)$, and one of the following holds: 
\begin{align}\label{sec56}
\limsup_{t \rar \infty} \| \nabla u(t) \|^2 < \frac{ 2N - 2(N-1)|\vl|^2 }{N \sqrt{1 - |\vl|^2} } \| \nW \|^2,  
\end{align}
or 
\begin{align}\label{sec57}
\limsup_{t \rar \infty} \left [ \| \nabla u(t) \|^2 + (N-1) \| \p_t u(t) \|^2 \right ] < \frac{2}{\sqrt{1 - |\vl|^2}} \| \nW \|^2, 
\end{align}
then there exist $x_0 \in \R^N$, $\iota_0 \in \{ \pm 1 \}$, and $\lam_0 > 0$ such that $u(t,x) = 
\iota_0 \lam_0^{(N-2)/2} W_{\vl} ( \lam_0 t, \lam_0 x + x_0 )$.  To see this, let $u$ be as above, $Q^+$ be given 
by Theorem \ref{thm510}.  It is is sufficient by Theorem \ref{thm511} to prove that $Q^+ = W$ up to sign, space translation, and 
scaling, since $W$ is non--degenerate.  Because of Theorem \ref{thm21} (see also the comments after Remark \ref{rmk59}), we are reduced to proving
that $\| \nabla Q^+ \|^2 < 2 \| \nW \|^2$.  Recall, from Section 2, that 
\begin{align}\label{sec58} 
\| Q^+ \|^2 = \| Q^+ \|^{2N/(N-2)}_{L^{2N/(N-2)}}, \quad \| \p_{x_j} Q^+ \|^2 = \frac{1}{N} \| \nabla Q^+ \|^2.  
\end{align}
By direct calculations, 
\begin{align}
\| \nabla Q^+_{\vl}(0) \|^2 &= \frac{N - (N-1) |\vl|^2}{N \sqrt{ 1 - |\vl|^2 }} \| \nabla Q^+ \|^2, \label{sec59} \\
\| \p_t Q^+_{\vl}(0) \|^2 &= \frac{|\vl|^2}{N \sqrt{ 1 - |\vl|^2 }} \| \nabla Q^+ \|^2, \label{sec510}. 
\end{align}
Thus, we see that \eqref{sec56} or \eqref{sec57}, together with (c) in Theorem \ref{thm510} imply that 
$\| \nabla Q^+ \|^2 < 2 \| W \|^2.$ Finally note that 
\begin{align*}
\inf_{0 \leq \ell < 1} \frac{2N - 2(N-1) \ell^2}{N \sqrt{1 - \ell^2}} = \frac{4 \sqrt{N-1}}{N}, 
\end{align*}
and thus Theorem \ref{thm511} implies Theorem \ref{thm57} and hence, also Theorem \ref{thm51}.  
\end{rmk}

\section{Concentration--compactness, II}

In this section, we describe a new compactness argument, due to the authors in \cite{15}, which can be used to describe the asymptotics of 
type II solutions of \eqref{nlw}.  The argument is very general and can be used for the same purpose for a variety of nonlinear 
dispersive equations.  In this section, we will consider solutions of \eqref{nlw} for which 
\begin{align*}
\sup_{0 \leq t < T_+(u)} \| (u(t), \p_t u(t) ) \|_{\energ} < \infty, 
\end{align*}
i.e. type II solutions.  As was shown in \cite{27}, \cite{26}, \cite{19}, and \cite{20}, there exist type II solutions for which $T_+(u) < \infty$, even in the 
regime near the ground state $W$.  There are also globally defined solutions, which don't scatter, such as the non--zero solutions
of the associated elliptic equation.  Even in the regime near the ground state $W$, there are solutions that don't scatter to 
a linear solution, or to $W$, see \cite{8}.  The soliton resolution conjecture, Conjecture \ref{conj55}, is a far--reaching 
attempt to describe the dynamics of type II solutions.  Our final goal is to establish Conjecture \ref{conj55} in all cases. 

The main result in this section is the following: 

\begin{thm}\label{thm61} 
Let $u$ be a type II solution of \eqref{nlw} and assume that $u$ does not scatter forward in time.  Then there exist 
sequences $\{ t_n \}_n$ in $[0,T_+(u))$, $\{\lam_n \}_n$ in $\R^+,$ an integer $J \geq 1$, and for all $j \in \{ 1, \ldots, J \}$, 
$Q^j \in \Sigma$, $\vl_j \in \R^N$, $|\vl_j | <1$, and a sequence $\{ x_{j,n} \}_n \in \R^N$, such that $\lim_n t_n = T_+(u)$, 
$1 \leq j < k \leq J \implies \lim_n \frac{|x_{j,n} - x_{k,n}|}{\lam_n} = \infty$, and 
\begin{enumerate}[(a)]
\item for all $T > 0$, $
\lam_n T + t_n < T_+(u)$ for $n$ large and 
\begin{align}\label{sec61}
\lim_{n \rar \infty} \int_0^T \int_{\R^N} \left |
\lam_n^{\frac{N-2}{2}} u(t_n + \lam_n t, \lam_n x) - \sum_{j = 1}^J Q^j_{\vl_j} \left ( t , x - \frac{x_{j,n}}{\lam_n} \right )
\right |^{\frac{2(N+1)}{N-2}} dx dt = 0,
\end{align}
\item for all $R > 0$, for all $j \in \{1, \ldots, J \}$, 
\begin{align}\label{sec62}
\lim_{n \rar \infty} \int_{|x| \leq R} \left |
\lam_n^{\frac{N}{2}} \nabla_{t,x} u(t_n, x_{j,n} + \lam_n x) - \nabla_{t,x} Q^j_{\vl_j} (0,x)
\right |^{2} dx = 0,
\end{align}
where $\nabla_{t,x} u = (\p_t u , \p_{x_1} u , \ldots, \p_{x_N} u )$, 
\item for all $j \in \{ 1, \ldots, J \}$, 
\begin{align*}
\left ( \lam_n^{\frac{N-2}{2}} u(t_n, \lam_n \cdot + x_{j,n}), 
\lam_n^{\frac{N}{2}} \p_t u(t_n, \lam_n \cdot + x_{j,n})  \right ) \wa_n \left ( Q^j_{\vl_j}(0), 
\p_t Q^j_{\vl_j}(0) \right ),
\end{align*}
weakly in $\energ$.     
\end{enumerate}
\end{thm}

\begin{rmk}\label{rmk62}
Note that (c) follows from \eqref{sec62}.  Also note that \eqref{sec62} implies by finite speed of propagation, that $\forall T > 0$, 
$\forall R > 0$, 
\begin{align}\label{sec63}
\lim_{n \rar \infty} \int_0^T \int_{|x| \leq R} \left |
\lam_n^{\frac{N}{2}} \nabla_{t,x} u(t_n + \lam_n t, x_{j,n} + \lam_n x) - \nabla_{t,x} Q^j_{\vl_j} (t,x)
\right |^{2} dx dt= 0,
\end{align}
Results of the type of Theorem \ref{thm61} (c) and \eqref{sec63} were first proved for wave maps in the equivariant 
setting by Christodoulou and Tahvildar--Zadeh \cite{5}, and Struwe \cite{32}, and for general wave maps by 
Sterbenz and Tataru \cite{31}.  These proofs depend crucially on the monotonicity of the wave map energy flux, which is not available for 
\eqref{nlw}.  
\end{rmk}

\begin{rmk}\label{rmk63}
Theorem \ref{thm61} is a consequence of Theorem \ref{thm510} and a very general minimality argument, based on 
profile decomposition.  The minimality argument is not specific to \eqref{nlw} but works for quite general dispersive 
equations for which a profile decomposition is available.  In that setting, one obtains that for any type II non--scattering solution, 
there exists a sequence of times converging to the final time of existence, such that the modulated solution converges 
(in some weak sense) to a compact solution.  See also the comment after Definition \ref{defn52}.   
\end{rmk}  

\begin{rmk}\label{rmk64}
The construction which yields Theorem \ref{thm61} can be seen as an extension of the \lq concentration--compactness' step, of the 
\lq concentration--compactness/rigidity theorem' method initiated in \cite{23} and \cite{24}, and explained in Section 4.  This approach
bypasses the construction of \lq critical elements' as, for instance, in Proposition \ref{ppn42} and Proposition \ref{ppn43}
and, as we will see in Section 7, it implies Theorem \ref{thm41} (i) and some extensions of it.  
\end{rmk}

\begin{rmk}\label{rmk65}
The proof of Theorem \ref{thm61} uses the pre--order on the set of profiles, introduced in Section 3, in Notation \ref{note314} and
subsequent results. 
\end{rmk}

We will now give a schematic description of the proof of Theorem \ref{thm61}.  One considers sequences $\{ t_n \}_n$, with $t_n \rar T_+(u)$ .  We then define: 
\begin{align}
\mathcal S_0 = \Bigl \{ \mbox{all sequences } & \{ t_n \}_n, t_n \rar T_+(u), \mbox{ such that } \{ (u(t_n), \p_t u(t_n) ) \}_n 
\nonumber \\ &
\mbox{ admits a well ordered profile decomposition} \Bigr \}. \label{sec64}
\end{align}
(See Definition \ref{defn315}.) In view of Claim \ref{clm316}, $\mathcal S_0$ contains, after extraction, every sequence 
$\{ t_n \}_n$. 


Next for $\tau = \{ t_n \}_n \in \mathcal S_0$, we define 
\begin{align}
J_0(\tau) = \mbox{ number of nonlinear profiles which do not scatter forward in time.} \label{sec65} 
\end{align}
This is well--defined because of Lemma \ref{lem310}, and given the Pythagorean expansion \eqref{sec37}, and the local theory 
of the Cauchy problem \eqref{nlw}, and the type II assumption, $J_0(\tau)$ has an upper bound independent of $\tau$.  
Moreover because of Proposition \ref{ppn317}, $J_0(\tau) \geq 1$, since $u$ does not scatter forward in time. 

A consequence of \eqref{sec65} is that, for $1 \leq j \leq J_0(\tau)$, the nonlinear profile $U^j$ does not scatter in 
forward time, while for $j \geq J_0(\tau) + 1$, $U^j$ scatters in forward time.  By the upper bound on $J_0(\tau)$, we can define
\begin{align}\label{sec66}
J_{\max} = \max \left \{ J_0(\tau) : \tau \in \mathcal S_0 \right \}, 
\end{align}
and we let 
\begin{align}\label{sec67}
S_1 = \left \{ \tau : J_0(\tau) = J_{\max} \right \}.
\end{align}
Clearly, $S_1 \neq \varnothing$.  Set, for $\tau \in \mathcal S_1$, 
\begin{align}\label{sec68}
\mathcal E(\tau) = \sum_{j = 1}^{J_{\max}} E( U^j, \p_t U^j).  
\end{align}
In view of Lemma \ref{lem310}, one can see that $\mathcal E(\tau)$ is well defined, i.e. it is independent of the profile decomposition 
of $\{ (u(\tau_n), \p_t u(\tau_n) ) \}_n$, for $\tau = \{ t_n \}_n$ in $\mathcal S_1$.  

Next, notice that, the type II assumption on $u$, and the Pythagorean expansion \eqref{sec38}, give a uniform lower bound for 
$\mathcal E(\tau)$, $\tau \in \mathcal S_1$.  Our minimization process consists of defining 
\begin{align}\label{sec69}
\mathcal E_{\min} = \inf \left \{ \mathcal E(\tau) : \tau \in \mathcal S_1 \right \}.  
\end{align}
A crucial step in \cite{15} is to show: 
\begin{align}\label{sec610}
\mathcal S_2 = \left \{ \tau \in \mathcal S_1 : \mathcal E(\tau) = \mathcal E_{\min} \right \} \neq \varnothing. 
\end{align}
Recall now the strict order $\prec$ from Notation \ref{note314}.  We then set, for $\tau \in \mathcal S_2$, 
\begin{align}\label{sec611}
J_1(\tau) = \min \left \{ j : j \prec j+1 \right \}.
\end{align}
Recall from our definition of $J_0(\tau)$, and the definition of the order that, since for $j \leq J_0(\tau)$, $U^j$ does not scatter, and 
for $j \geq J_0(\tau) + 1$, $U^j$ does scatter, $J_1(\tau) \leq J_0(\tau)$.  We now define 
\begin{align}
J_{\min} &= \min \left \{ J_1(\tau) : \tau \in \mathcal S_2 \right \}, \label{sec612} \\
\mathcal S_3 &= \left \{ \tau \in \mathcal S_2 : J_1(\tau) = J_{\min} \right \}.
\end{align}
Since $\mathcal S_2$ is non--empty, clearly so is $\mathcal S_3$.  After this, in \cite{15} one shows that there exists $\tau \in \mathcal S_3$ such that, for all $j \in \{1,\ldots, J_{\min} \}$, $U^j$ is a compact solution, $\lam_{j,n} = \lam_n$, $j = 1, \ldots, J_{\min}$, 
$t_{j,n} = 0$, $j = 1, \ldots, J_{\min}$. This is the result that holds for \lq general ' dispersive equations.  
Using that for \eqref{nlw}, we have Theorem \ref{thm510}, one then shows: 
\begin{lem}\label{lem66}
There exists $\tau \in \mathcal S_3$, such that for $j =1, \ldots, J_{\min}$, $\exists \vl_j$, $| \vl_j | < 1$, and $Q^j \in \Sigma$, such that $\left (U^j_L(0), \p_t U^j_L(0) \right )= \left ( Q^j_{\vl_j}(0), \p_t Q^j_{\vl_j}(0) \right )$, $t_{j,n} \equiv 0$, $T_+(U^j) = \infty$, 
$\lam_{1,n} = \lam_{2,n} = \cdots = \lam_{J_{\min},n} = \lam_n$, and, $\forall R > 0$
\begin{align}\label{sec614}
\lim_{n \rar \infty} \int_{|x| \leq R} \left |
\lam_n^{\frac{N}{2}} \nabla_{t,x} u(t_n, x_{j,n} + \lam_n x) - \nabla_{t,x} Q^j_{\vl_j} (0,x)
\right |^{2} dx = 0, \quad 1 \leq j \leq J_{\min},
\end{align}
and 
\begin{align}\label{sec615}
\left ( \lam_n^{\frac{N-2}{2}} u(t_n, \lam_n \cdot + x_{j,n}), 
\lam_n^{\frac{N}{2}} \p_t u(t_n, \lam_n \cdot + x_{j,n})  \right ) \wa_n \left ( Q^j_{\vl_j}(0), 
\p_t Q^j_{\vl_j}(0) \right ),
\end{align}
weakly in $\energ$.     
\end{lem}

From Lemma \ref{lem66} and further changes on the sequence of times $\{ t_n \}_n$, we then obtain the full form of Theorem 
\ref{thm61}.  

In the coming sections, instead of the full form of Theorem \ref{thm61}, we will use Lemma \ref{lem66}.

\section{An extension of Theorem \ref{thm41} (i) and concentration estimates for type II solutions}

In this section we will prove, using Lemma \ref{lem66}, an extension of Theorem \ref{thm41} (i), first stated in \cite{10}.  We will then give some general properties of type II solutions which blow--up in finite time, which first appeared in \cite{9}, and prove concentration estimates for such solutions, which originate in Section 7 of \cite{24} and in \cite{9}.  The extension of Theorem \ref{thm41} (i) is: 

\begin{thm}\label{thm71}
Assume that $3 \leq N \leq 5$.  Let $u$ be a solution of \eqref{nlw} which satisfies 
\begin{align}\label{sec71}
\limsup_{t \rar T_+(u)} \left [ \| \nabla u(t) \|^2 + \frac{N-2}{2} \| \p_t u(t) \|^2 \right ] < \| \nW \|^2. 
\end{align}
Then $T_+(u) = \infty$ and $u$ scatters in forward time.  If $u$ is radial, \eqref{sec71} can be replaced by 
\begin{align}\label{sec72}
\limsup_{t \rar T_+(u)} \| \nabla u(t) \|^2 < \| \nabla W \|^2.
\end{align}
\end{thm}

We will make a few remarks before proceeding to the proof.  

\begin{rmk}\label{rmk72} 
If $u$ is a solution of \eqref{nlw} such that $E(u_0,u_1) < \EW$ and $\| \nabla u_0 \| < \| \nabla W \|$, then \eqref{sec71}
holds.  In fact, since by the variational estimate \eqref{sec213}, we have, for $t \in (T_-(u), T_+(u))$, $E(u(t), \p_t u(t)) < \EW$ and 
$\| \nabla u(t) \| < \| \nW \|$, it suffices to show that $\| \nabla u_0 \|^2 + \frac{N-2}{2} \| \nabla u_1 \|^2 
< \| \nabla W \|^2$.  We have $\frac{1}{2} \| u_1 \|^2 < \EW - E(u_0,0)$.  Also the variational estimate \eqref{sec215} gives that 
$\| \nabla u_0 \|^2 \leq N E(u_0,0)$, so that 
\begin{align*}
\| \nabla u_0 \|^2 +  \frac{N-2}{2} \| u_1 \|^2 &< N E(u_0,0) + (N-2) \EW - (N-2) E(u_0,0) \\
&= 2 E(u_0,0) + (N-2) \EW \\
&< 2 \EW + (N-2) \EW = N\EW = \| \nabla W \|^2,
\end{align*}
by \eqref{sec26}.  We thus see that Theorem \ref{thm71} is an extension of Theorem \ref{thm41} (i).  
\end{rmk}

\begin{rmk}\label{rmk73}
Theorem \ref{thm71} is exactly Corollary 1.5 in \cite{10}.   As was pointed out in \cite{10}, this corrected Corollary 7.4 in \cite{24}, in the non--radial 
setting.  In Corollary 7.4 in \cite{24} (stated without proof), \eqref{sec71} was replaced by 
\begin{align}\label{sec73}  
\limsup_{t \rar T_+(u)} \left [ \| \nabla u(t) \|^2 + \| \p_t u(t) \|^2 \right ]  < \| \nW \|^2.   
\end{align}
Note that \eqref{sec73} is stronger than \eqref{sec71} for $N = 3$, coincides with it for $N = 4$, and is weaker than 
\eqref{sec71} when $N = 5$.  It is not difficult to see, as was pointed out in \cite{10}, that Corollary 7.4 in \cite{24} fails for $N = 5$.  
In fact, consider the solution $W_\ell(x) = W \left ( \frac{x_1 - t\ell}{\sqrt{1 - \ell^2}}, x' \right )$, introduced in \eqref{sec54}.  Then, a calculation shows that $\| \nabla W_\ell(t) \|^2 = \frac{N + (1 - N)\ell^2}{N \sqrt{1 - \ell^2}} \| \nabla W \|^2$, 
$\| \p_t W_\ell(t) \|^2 = \frac{\ell^2}{N \sqrt{1 - \ell^2}} \| \nabla W \|^2$, so that 
\begin{align*}
\| \nabla W_\ell(t) \|^2 + \| \p_t W_\ell(t) \|^2 = \frac{N + (2 - N)\ell^2}{N \sqrt{1 - \ell^2}} \| \nabla W \|^2, 
\end{align*}
and one can see that for $N = 5$, $\ell$ small we have $\frac{N + (2-N)\ell^2}{N \sqrt{ 1 - \ell^2}} < 1$, and so 
$W_\ell$ verifies \eqref{sec73} for $N = 5$, but does not scatter, being a traveling wave solution.  
\end{rmk}

The proof of Theorem \ref{thm71} given in \cite{10} (Corollary 1.5 in \cite{10}) has a gap, since it uses the false Pythagorean expansions 
\eqref{sec39} and \eqref{sec310}.  The reason for needing \eqref{sec39} and \eqref{sec310} in the proof given in \cite{10}
is that the coefficients in front of the spatial and time derivatives in \eqref{sec71} are different.  This arises in the construction of the critical element.  If one assumes, for $N = 3,4$, that the condition \eqref{sec73} holds instead, then the proof given in \cite{10} goes 
through, using the correct Pythagorean expansion \eqref{sec37}.  This fails for $N = 5$ (as it has to), since, as was
shown above, there are non--zero compact solutions verifying \eqref{sec73}.  To close the gap in the proof of Corollary
1.5 in \cite{10}, we will now give a proof of Theorem \ref{thm71} based on Lemma \ref{lem66} (whose proof only uses the Pythagorean 
expansion \eqref{sec37} ) and avoids the use of critical elements. 

\begin{proof}[Proof of Theorem \ref{thm71}]
We argue by contradiction.  Assume that $u$ verifies \eqref{sec71} and does not scatter.  Because of \eqref{sec71}, we can apply the results of Section 6 (note that the type II assumption is only needed near $T_+(u)$).  We now use Lemma 
\ref{lem66}, which provides us with a sequence of times $t_n \rar T_+(u)$, $\{ \lam_n \}_n$ in $\R^+$, $\{ x_n \}_n$ in $\R^N$, 
$\vl \in \R^N$, $|\vl | < 1$, and $Q \in \Sigma$ so that \eqref{sec614} holds with $Q^j_{\vl_j} = Q_{\vl}$.  Because of 
\eqref{sec71} and \eqref{sec614}, we obtain that 
\begin{align}\label{sec74}
\| \nabla Q_{\vl}(0) \|^2 + \frac{N-2}{2} \| \p_t Q_{\vl}(0) \|^2 < \| \nW \|^2.
\end{align}
After a rotation of $\R^N$, $Q_{\vl}(t,x) = Q_{\ell \vec e_1}(t,x) = Q \left ( \frac{x_1 - t \ell}{\sqrt{1 - \ell^2}}, x' \right )$, 
$-1 < \ell < 1$.  Recall from \eqref{sec210} that for $Q \in \Sigma$, $\int | \p_{x_j} Q |^2 dx = \frac{1}{N} \int |\nabla Q|^2 dx,$, 
$1 \leq j \leq N$.  A calculation gives: 
\begin{align*}
\int |\p_t Q_\ell |^2 dx &= \frac{\ell^2 \sqrt{ 1 - \ell^2}}{1 - \ell^2} \int |\p_{x_1} Q|^2 dx, \\
\int |\nabla Q_{\ell} |^2 dx &= \frac{\sqrt{1 - \ell^2}}{1 - \ell^2} \int |\p_{x_1} Q |^2 dx + 
\sqrt{ 1 - \ell^2} \int |\nabla_{x'} Q |^2 dx.  
\end{align*}
Hence, 
\begin{align*}
\int |\nabla Q_{\ell} |^2 dx + \frac{N-2}{2} \int |\p_t Q_{\ell} |^2 dx 
&= \frac{1 - \ell^2 /2 }{\sqrt{1 - \ell^2}} \int |\nabla Q |^2 dx, \\
\int |\nabla Q_{\ell} |^2 dx + \frac{N-2}{2} \int |\p_t Q_{\ell} |^2 dx - \int |\nabla Q |^2 dx
&= \frac{1 - \ell^2 /2 - \sqrt{ 1 - \ell^2}}{\sqrt{1 - \ell^2}} \int |\nabla Q |^2 dx. 
\end{align*}
Now use, as in \cite{10} Claim 2.5, the standard inequality $\sqrt{ 1 - x} \leq 1  - x/2 - x^2/ 8$ for $0 \leq x < 1$,  with $x = \ell^2$, 
to get 
\begin{align}\label{sec75}
\int |\nabla Q_{\ell} |^2 dx + \frac{N-2}{2} \int |\p_t Q_{\ell} |^2 dx 
&\geq \left ( 1 + \frac{\ell^4}{8} \right ) \int |\nabla Q |^2 dx.
\end{align}
Combining \eqref{sec74} and \eqref{sec75} yields 
\begin{align}\label{sec76}
\left ( 1 + \frac{\ell^4}{8} \right ) \int |\nabla Q |^2 dx < \int |\nabla W |^2 dx. 
\end{align}
But this, by Theorem \ref{thm21}, gives that $Q = \pm W_{\lam_0} ( \cdot + x_0)$, which contradicts \eqref{sec76}, as desired, showing that $u$ must scatter forward in time.  In the case when $u$ is radial and \eqref{sec72} holds, using also the conservation of energy, we
deduce that $u$ is type II near $T_+(u)$.  We assume as before that $u$ does not scatter forward in time and apply 
Lemma \ref{lem66} and observe that, since $u$ is radial, so is $Q_{\vl}$, which forces first $\vl = 0$, and second $Q = \pm W_{\lam_0}$, since those are the only radial solutions of the elliptic problem (see the line after \eqref{sec26}).  In light of \eqref{sec72} and Lemma \ref{lem66}, this forces $\| \nabla W \|^2 < \| \nabla W \|^2$, which is the desired contradiction in this case.   Thus, Theorem 
\ref{thm71} is proved.   
\end{proof}

We now turn to our description of general type II blow--up solutions.  

\begin{defn}\label{defn74}
Let $u$ be a type II blow--up solution, i.e. $T_+(u) < \infty$ and
\begin{align*}
\sup_{0 < t < T_+(u)} \left [ \| \nabla u(t) \|^2 + \| \p_t u(t) \|^2 \right ] < \infty. 
\end{align*}
Let $x_0 \in \R^N$.  We say that $x_0$ is regular if $\forall \e > 0$, $\exists R>0$ such that $\forall t \in [0, T_+(u))$
\begin{align*}
\int_{|x - x_0| < R} |\nabla u(t)|^2  + |\p_t u(t) |^2 + \frac{|u(t)|^2}{|x - x_0|^2} dx \leq \e. 
\end{align*}
If $x_0$ is not regular, we say that it is singular.  
\end{defn}

We next have the following result from \cite{9}. 

\begin{thm}\label{thm75}
Let $u$ be a type II blow--up solution of \eqref{nlw}, and $T_+ = T_+(u) < \infty$.  Then, there exist $J \in \N \backslash \{ 0 \}$
and $J$ distinct points $x_1, \ldots, x_J$ of $\R^N$ such that $S = \{ x_1, \ldots, x_J \}$ is the set of singular points of $u$.  
Furthermore, there exists $(v_0,v_1) \in \energ$ such that $(u(t), \p_t u(t)) \wa_{t \rar T_+} (v_0,v_1)$ weakly in $\energ$.
  If $\varphi \in C^\infty_0(\R^N)$ is equal to 1 around each singular point, then 
\begin{align}\label{sec77}
\lim_{t \rar T_+} \left \| \left (
(1 - \varphi) ( u(t) - v_0 ) , ( 1 - \varphi) (\p_t u(t) - v_1 ) \right ) \right \|_{\energ} = 0. 
\end{align} 
\end{thm}

We refer to \cite{4} for the proof of this result.  The main ingredients of the proof are the type II bound on $u$, finite speed 
of propagation, and the local theory of the Cauchy problem as in Section 2.  We note also that \cite{27}, \cite{19}, and \cite{20} give 
examples of radial type II blow--up solutions, and non--radial examples can be constructed from these, using Lorentz transformations.  

\begin{defn}\label{defn76}
Let $u$, $(v_0,v_1)$ be as in Theorem \ref{thm75}.  Let $v$ be the solution of \eqref{nlw} such that $(v(T_+), \p_t v(T_+) )
= (v_0,v_1)$.  We will call $v$ the regular part of $u$ at the blow--up time $T_+$, and $a = u - v$, the singular part of $u$.  
Note that \eqref{sec77} and finite speed of propagation imply that 
\begin{align}\label{sec78s}
\supp (a, \p_t a) \subset \bigcup_{j = 1}^J \{ (t,x) : |x - x_j| \leq |t - T_+| \}. 
\end{align}
\end{defn}

We now turn to our concentration results for type II finite time blow--up solutions.  


\begin{thm}\label{thm77}
Let $3 \leq N \leq 5$.  Let $u$ be a type II blow--up solution of \eqref{nlw} as in Theorem \ref{thm75}.  Then, for each 
$j \in \{ 1, \ldots, J \}$, we have 
\begin{align}\label{sec78}
\liminf_{t \rar T_+} \int_{| x - x_j | \leq | t - T_+ |}
|\nabla u(t,x) |^2 + |\p_t u(t,x)|^2 dx \geq \frac{2}{N} \int |\nabla W |^2 dx,  
\end{align}
and 
\begin{align}\label{sec79}
\limsup_{t \rar T_+} \int_{| x - x_j | \leq | t - T_+ |}
|\nabla u(t,x) |^2 + \frac{N-2}{2} |\p_t u(t,x)|^2 dx \geq \int |\nabla W |^2 dx.  
\end{align}
\end{thm}

\begin{rmk}\label{rmk78} 
A proof of a weaker version of \eqref{sec79} is sketched in Corollary 7.5 of \cite{10}.  A full proof of \eqref{sec78} is given by the authors in 
\cite{9}.  In Corollary 7.5 of \cite{24}, a different form of \eqref{sec79} is stated for $3 \leq N \leq 5$, with $\frac{N-2}{2}$ replaced by 1, 
as in \eqref{sec73}.  The sketch of the proof of that, given in \cite{24}, relied on the fact that \eqref{sec73} implies forward scattering, which, as we saw earlier, is not true in $N = 5$, and therefore a gap remained in the proof of Corollary 7.5 of \cite{24}.  This gap was 
closed for $N = 3,4$ in \cite{9}, where the proof of \eqref{sec79} with $\frac{N-2}{2}$ replaced by 1 was given, and this relied only on the correct Pythagorean expansion \eqref{sec37}.  In \cite{10}, Corollary 1.7, the correct form of \eqref{sec79} is stated for 
$N = 3,4,5$ (in a slightly weaker form) with the coefficient $\frac{N-2}{2}$ in front of the $|\p_t u(t)|^2$ term.  The proof of this 
was not given, but it implicitly relied on the incorrect Pythagorean expansions \eqref{sec39} and \eqref{sec310} and thus, 
a gap remained in that proof, which we will now fill, using Theorem \ref{thm71} and \eqref{sec311}, \eqref{sec312}.
\end{rmk}

\begin{proof}[Proof of Theorem \ref{thm77}]
We start out with a sketch of the proof \eqref{sec78}, following the proof of (3.5) in Theorem 3.2 of \cite{9}.  Thus, we choose an 
increasing sequence $\{t_n \}_n$ which tends to $T_+$, which, without loss of generality, we can assume to be 1, and $t_n > t_0$,
where $t_0 > T_-(v)$ and $v$ is defined in Definition \ref{defn76}.  We also assume, without loss of generality, that $0 \in S$, and 
choose $\psi \in C^\infty_0 (\R^N)$, $\supp \psi \cap S = \{ 0 \}$, and $\psi \equiv 1$ near 0.  After extraction, let 
$\left ( U^j_L, \{ \lam_{j,n}, x_{j,n}, t_{j,n} \}_n \right )_j$ be a profile decomposition of $\{ ( \psi ( u(t_n) - v(t_n) ), 
\psi (\p_t u(t_n) - \p_t v(t_n) ) ) \}_n$.  Note that since $\psi (u - v)$ is supported in $\{ |x| \leq 1 - t \}$, by \eqref{sec78s} and 
our choice of $\psi$, because of \cite{1}, page 154--155, we have 
\begin{align}\label{sec710} 
\forall j \geq 1, \exists C_j \mbox{ such that, } \forall n \quad |\lam_{j,n}| + |x_{j,n}| + |t_{j,n}| \leq C_j (1 - t_n).
\end{align}
One then reorders the decomposition, in such a way that 
\begin{align*}
\| \nabla U_0^1 \|^2 + \| U^1_1 \|^2 = \sup_{j \geq 1} \left \{ \| \nabla U^j_0 \|^2 + \| U^j_1 \|^2 \right \},
\end{align*}
(note that this is finite because of \eqref{sec37}).  Then, 
\begin{align}\label{sec711}
\| \nabla U^1_0 \|^2 + \| U^1_1 \|^2 \geq \frac{2}{N} \| \nabla W \|^2.  
\end{align}
Note that once \eqref{sec711} is established, the Pythagorean identities \eqref{sec37}, \eqref{sec38}, the fact that $ 
\supp ( u(t_n) - v(t_n) , \p_t u(t_n) - \p_t v(t_n) ) \subset \{ |x| \leq |1 - t_n| \}$, the facts that 
$\psi = 1$ near $0$,  and $v$ is regular at $t =1$, give \eqref{sec78}.  If \eqref{sec711} does not hold, then for all $j \geq 1$, we have 
\begin{align*}
\| \nabla U_0^j \|^2 + \| U^j_1 \|^2 < \frac{2}{N} \| \nabla W \|^2.  
\end{align*}
Using that $2 E(f,g) \leq \| \nabla f \|^2 + \| g \|^2$, and that $\EW = \frac{1}{N} \| \nabla W \|^2$, we find that there exists 
$\e_0 > 0$ such that, for all $j,n$
\begin{align*}
E \left ( U^j_L ( -t_{j,n} / \lam_{j,n} ), \p_t U^j_L ( -t_{j,n} / \lam_{j,n} ) \right ) &\leq \EW - \e_0,  \\
\| \nabla U^j_L ( -t_{j,n} / \lam_{j,n} ) \|^2 &\leq \| \nabla W \|^2 - \e_0.
\end{align*}
Then, by Theorem \ref{thm41}, all the nonlinear profiles $U^j$ scatter and by Proposition \ref{ppn317}, the solution with initial 
condition $ ( \psi( u(t_n) - v(t_n) ) , \psi ( \p_t u(t_n) - \p_t v(t_n) ) )$ is globally defined and scatters for large $n$.  Next, note that 
by Remark \ref{rmk311}, and Theorem \ref{thm75}, we have that $\left (\psi v(t_n), \{ U^j_L \}_{j\geq 1} \right )$ is a profile 
decomposition of $\psi u(t_n)$, with $\lam_{0,n} = 1$, $t_{0,n} = 0$, $x_{0,n} = 0$, in light of \eqref{sec710}.  But then, 
$\psi u(t_n)$ is now defined beyond $t = 1$, by Theorem \ref{thm312} and that fact that $v$ is a regular solution at $t = 1$.  This 
contradicts the fact that $0$ is a singular point of $u$, by finite speed of propagation.  

Next we turn to the proof of \eqref{sec79}.  We will establish that, given $\e_0 > 0$, there exist
$\{ \tilde{t}_n \}_n$, $t_0 < \tilde t_n < 1$, $\tilde t_n \rar 1$, such that $ \{ ( \psi( u(\tilde t_n) - v(\tilde t_n) ) , \psi ( \p_t u(
\tilde t_n) - \p_t v(\tilde t_n) ) ) \}_n$ has a profile decomposition $\left ( \tilde U^j_L, \{ \tilde \lam_{j,n}, \tilde x_{j,n}, \tilde t_{j,n} \}_n \right )_j$ such that $\tilde t_{1,n} \equiv 0$ and 
\begin{align*}
\| \nabla \tilde U^1_0 \|^2 + \frac{N-2}{2} \| \tilde U^1_1 \|^2 \geq \| \nabla W \|^2 - \e_0.
\end{align*}
Once this is established, applying \eqref{sec311} and \eqref{sec312}, using the support of $(u - v, \p_t u - \p_t v)$ and the fact 
that $v$ is a regular solution at $t = 1$, gives \eqref{sec79}.  We now turn to the proof of the statement above.  We 
will do extractions systematically, without further comment.  Let $t_n \rar 1$, $t_n > t_0$.  Let $\tilde u_n$, $\tilde v_n$ be the 
solutions of \eqref{nlw}, with 
\begin{align*}
(\tilde u_n, \p_t \tilde u_n) |_{t = t_n} = ( \psi u(t_n), \psi \p_t u(t_n) ), \\
(\tilde v_n, \p_t \tilde v_n) |_{t = t_n} = ( \psi v(t_n), \psi \p_t v(t_n) ).
\end{align*}
By finite speed of propagation, and the fact that $0$ is a singular point for $u$, $T_+(\tilde u_n) \leq 1$.  Also, since 
$( \psi v(t_n), \psi \p_t v(t_n))$ has a limit in $\energ$ as $n \rar \infty$, there exists  a small $\tilde t_0 > 0$ such that $
\tilde v_n(t + t_n)$ is well--defined for large $n$, $|t| \leq \tilde t_0$.  We now take a profile decomposition
$\left ( U^j_L, \{ \lam_{j,n}, x_{j,n}, t_{j,n} \}_n \right )_j$ for the sequence $( \tilde u_n(t_n) - \tilde v_n(t_n), 
\p_t \tilde u_n(t_n) - \p_t \tilde v_n(t_n) )$.  By the support property of $(u - v)(t)$, we again have 
\begin{align*}
\forall j \geq 1, \forall n \geq 1, \quad |\lam_{j,n}| + |x_{j,n}| + |t_{j,n}| \leq C_j ( 1 - t_n ).
\end{align*} 
We consider the associated nonlinear profiles $U^j$, and we reorder the profiles so that, for some $J_0$, for 
$1 \leq j \leq J_0$, $\| U^j \|_{S(0,T_+(U^j))} = \infty$, and for $j \geq J_0 + 1$, $\| U^j \|_{S(0,T_+(U^j))} < \infty$.  
Note that $J_0 \geq 1$ by Proposition \ref{ppn317} and the fact that $0$ is a singular point, as in the proof of 
\eqref{sec711}.  Also recall that for $j \geq J_0 + 1$, $T_+(U^j) = \infty$, by the definition of $J_0$ and the finite time blow--up
criterion \eqref{sec23}.  For $1 \leq j \leq J_0$, let $\ell_j = \lim_n - t_{j,n} / \lam_{j,n} \in \{ -\infty \} \cap  \R$.  The case $\ell_j = \infty$ 
is excluded because $U^j$ does not scatter forward, see the comments after \eqref{sec322}.  If $|\ell_j | < \infty$, by a time shift and changing the profile $U^j_L$, we can assume, without loss of generality, that $t_{j,n} \equiv 0$ (see Lemma \ref{lem39}).
Thus, for $1 \leq j \leq J_0$, either $t_{j,n} \equiv 0$, or $\lim_n -t_{j,n} / \lam_{j,n} =-\infty$.  Since $U^j$, $1 \leq j \leq 
J_0$ does not scatter forward, by Theorem \ref{thm71}, for any $\e_0 > 0$, there exists $T_j$ such that 
$T_-(U^j) < T_j < T_+(U^j)$ and 
\begin{align}\label{sec712} 
\| \nabla U^j(T_j) \|^2 + \frac{N-2}{2} \| \p_t U^j(T_j) \|^2 \geq \| \nabla W \|^2 - \e_0.
\end{align} 
Note that, if $t_{j,n} \equiv 0$, then $T_+(U^j) > 0$ and we can then choose $T_j > 0$, because $U^j$ does not scatter forward.
After extracting and reordering the first $J_0$ profiles, we can assume that, for all $n$
\begin{align}\label{sec713}
t_{1,n} + \lam_{1,n} T_1 = \min_{1 \leq j \leq J_0} ( t_{j,n} + \lam_{j,n} T_j ).
\end{align}
Let $\theta_n = t_{1,n} + \lam_{1,n} T_1$, and note that $\theta_n > 0$ for large $n$.  Moreover, our bounds on $
(\lam_{1,n}, x_{1,n}, t_{1,n} )$ show that $\theta_n \rar_n 0$.  By the definition of $\theta_n$, for all 
$j \in \{ 1, \ldots, J_0 \}$, we have $\frac{\theta_n - t_{j,n}}{\lam_{j,n}} \leq T_j < T_+(U^j)$.  We can then 
apply Theorem \ref{thm312}.  Thus, $t_n + \theta_n < T_+(\tilde u_n) \leq 1$ and $\left \{ \| \tilde u_n \|_{S(t_n, t_n + \theta_n)}
\right \}_n$ is bounded (once again, as in the proof of \eqref{sec711}, $\left ( \psi v(t_n), \{ U^j_L \}_{j \geq 1} \right )$
is a profile decomposition for $\tilde u(t_n)$, with $\lam_{0,n} = 1, t_{0,n} = 0, x_{0,n} = 0$), and 
\begin{align*}
\tilde u_n(t_n + t) &= \tilde v_n(t_n + t) + 
\sum_{j = 1}^J \frac{1}{\lam_{j,n}^{(N-2)/2}} U^j \left ( \frac{t- t_{j,n}}{\lam_{j,n}}, \frac{x - x_{j,n}}{\lam_{j,n}} \right )
+ w^J_n(t,x) + r_n^J(t,x), \\
\p_t \tilde u_n(t_n + t) &= \p_t \tilde v_n(t_n + t) + 
\sum_{j = 1}^J \frac{1}{\lam_{j,n}^{(N-2)/2}} \p_t U^j \left ( \frac{t- t_{j,n}}{\lam_{j,n}}, \frac{x - x_{j,n}}{\lam_{j,n}} \right )
+ \p_t w^J_n(t,x) + \p_t r_n^J(t,x) 
\end{align*}
for $0 < t < \theta_n$.  As in the proof of Corollary \ref{cor313}, this provides a profile 
decomposition for $(\tilde u_n (\tilde t_n), \p_t \tilde u_n(\tilde t_n)) - (\tilde v_n( \tilde t_n), \p_t \tilde v_n (\tilde t_n))$, where 
$\tilde t_n = t_n + \theta_n$.  Notice that, by finite speed of propagation, $\exists r_0 > 0$ such that, for large $n$, we have 
\begin{align*}
(\tilde u_n(\tilde t_n) , \p_t \tilde u_n( \tilde t_n)  ) &= ( u(\tilde t_n), \p_t u(\tilde t_n)), \quad |x| \leq r_0, \\
(\tilde v_n(\tilde t_n) , \p_t \tilde v_n( \tilde t_n)  ) &= ( v(\tilde t_n), \p_t v(\tilde t_n)), \quad |x| \leq r_0.
\end{align*}
Since $\supp ( u(\tilde t_n) - v ( \tilde t_n ) , \p_t u(\tilde t_n) - \p_t v( \tilde t_n) ) \subset \{ |x| \leq 1 - \tilde t_n \}$, 
we can replace in the above decomposition, $\tilde u_n$ and $\tilde v_n$ by $\psi u, \psi v$.  Finally 
$\frac{\theta_n - t_{1,n}}{\lam_{1,n}} = T_1$ and hence $\tilde t_{1,n} = -T_1 \lam_{1,n}$, and thus, after changing the first
profile, we can assume that $\tilde t_{1,n} \equiv 0$.  By the definition of $T_1$, we have  
\begin{align*}
\| \nabla U^1(T_1) \|^2 + \frac{N-2}{2} \| \p_t U^1(T_1) \|^2 \geq \| \nabla W \|^2 - \e_0, 
\end{align*}
which, by the definition of the new first profile (see the proof of Corollary \ref{cor313} and Lemma \ref{lem39}) finishes the proof 
of \eqref{sec79}.  
\end{proof}

\begin{cor}\label{cor79}
There exists $\eta_0 > 0$ such that , if $3 \leq N \leq 5$, 
\begin{align}\label{sec714}
\limsup_{t \rar T_+(u)} \left [ \| \nabla u(t) \|^2 + \frac{N-2}{2} \| \p_t u(t) \|^2 \right ] < \| \nabla W \|^2 + \eta_0,
\end{align}
and $u$ is a type II blow--up solution, then the singular set $S$ contains only one point.  
\end{cor}

\begin{proof}
Assume that there exist $x_1 \neq x_2$ in the singular set $S$.  Let $\e_0 > 0$ be given, and choose a sequence of times $t_n \rar T_+$ such that, for large $n$, 
\begin{align}\label{sec715}
\int_{|x - x_1| \leq |T_+ - t_n|} | \nabla u(t_n, x) |^2 dx + \frac{N-2}{2} \int_{|x - x_1| \leq |t_n - T_+|} 
|\p_t u(t_n,x) |^2 dx \geq \int |\nabla W |^2 dx - \e_0, 
\end{align}
as we can from \eqref{sec79}.  Then, apply \eqref{sec78}, to conclude that, for large $n$, we have, for the same 
sequence of times $\{t_n\}_n$, 
\begin{align}\label{sec716}
\int_{|x - x_2| \leq |T_+ - t_n|} | \nabla u(t_n, x) |^2 dx + \int_{|x - x_2| \leq |t_n - T_+|} 
|\p_t u(t_n,x) |^2 dx \geq \frac{2}{N} \int |\nabla W |^2 dx - \e_0.  
\end{align}  
Notice that, since $|x_1 - x_2 | > 0$, for large $n$, 
\begin{align*}
\{ |x - x_2| < |T_+ - t_n| \} \cap \{ |x - x_1| < |T_+ - t_n| \} = \varnothing. 
\end{align*}
Consider first $N = 3,4$, so that $\frac{N-2}{2} \leq 1$.  Then from \eqref{sec715} and \eqref{sec716}, letting $\e_0 \rar 0$
we see that 
\begin{align*}
\int |\nabla u(t_n) |^2 dx + \frac{N-2}{2} |\p_t u(t_n)|^2 dx \geq \left ( 1 + \frac{N-2}{N} \right ) \int |\nW |^2 dx.  
\end{align*}
Hence this contradicts \eqref{sec714} if $\eta_0 = \frac{N-2}{N} \int |\nW |^2 dx$.  Next consider $N =5$, so that 
$\frac{N-2}{2} > 1$.  Then, from \eqref{sec715} and \eqref{sec716}, letting $\e_0 \rar 0$, we see that 
\begin{align*}
\int |\nabla u(t_n,x) |^2 + \frac{N-2}{2} |\p_t u(t_n,x) |^2 dx \geq \left ( 1 + \frac{2}{N} \right ) \int |\nW|^2 dx, 
\end{align*}
which contradicts \eqref{sec714} if $\eta_0 = \frac{2}{N} \int |\nW|^2 dx$.  
\end{proof}

\begin{rmk}\label{rmk710}
In the papers \cite{27}, \cite{19}, and \cite{20}, radial solutions which are type II blow--up solutions are constructed, which verify \eqref{sec714}
for any $\eta_0 > 0$.  Thus, the situation described in Corollary \ref{cor79} is not vacuous.  Using Lorentz transformations, non--radial examples can also be constructed. 
\end{rmk}

\section{Universality of the blow--up profile for small type II solutions in the non--radial case}

In this section we give a precise description of solutions verifying \eqref{sec714} with $T_+ < \infty$, for $\eta_0$ small, in the 
non--radial setting.  The main result, which is valid in dimension $N = 3, 5$, is that such solutions, near the singular point, are 
a sum of a rescaled $W_\ell$, for suitable, small $\ell$, and a radiation term $(v_0,v_1) \in \energ$, plus a term which converges
to 0 in $\energ$.  Thus, $W_\ell$ is a universal profile.  Moreover, this is precisely a decomposition of the type in Conjecture \ref{conj55}, when the solution is \lq near W.'  A key ingredient in the proof is this result, by the authors, in \cite{9}, \cite{10}, 
is the use of \lq channels of energy,' which are obtained from exterior energy estimates for the linear wave equation in odd dimensions.  

We start out by recalling the exterior energy estimates.  
\begin{ppn}[\cite{10}] \label{p81} 
Assume that $N \geq 3$ is odd.  Let $u_0 \in \dot H^1$, $u_1 \in L^2$, and let $u_L$ be the solution to 
\begin{align*}
\left\{
     \begin{array}{lr}
       \partial_t^2 u_L - \Delta u_L = 0, \quad (t,x) \in \R \times \R^N, \\
       u_L|_{t = 0} = u_0, \quad \partial_t u_L|_{t = 0} = u_1,
     \end{array}
   \right..
\end{align*} 
Then, $\forall t \geq 0$
\begin{align}\label{s81}
\int_{|x| \geq t} |\nabla u_L(t)|^2 + |\p_t u_L(t) |^2 dx \geq \frac{1}{2} \left [ \| \nabla u_0 \|^2 + \| u_1 \|^2 \right ],
\end{align}
or $\forall t \leq 0$
\begin{align}\label{s82}
\int_{|x| \geq -t} |\nabla u_L(t)|^2 + |\p_t u_L(t) |^2 dx \geq \frac{1}{2} \left [ \| \nabla u_0 \|^2 + \| u_1 \|^2 \right ].
\end{align}
\end{ppn}

For the proof of this result, we refer to \cite{10}.  

\begin{rmk}\label{r82}
Proposition \ref{p81} fails for even dimensions, even in the radial case, as was shown in \cite{7}.  This is responsible for the fact that we can prove our 
main result only in odd dimensions.  
\end{rmk}

The main result of this section is: 

\begin{thm}\label{t83}
Assume that $N = 3$ or $N = 5$ and let $\eta_0 > 0$ be a small parameter.  Let $u$ be a solution to \eqref{nlw} such that 
$T_+(u) < \infty$, and 
\begin{align}\label{s83}
\limsup_{t \rar T_+(u)} \left [ \| \nabla u(t) \|^2 + \frac{N-2}{2} \| \p_t u(t) \|^2 \right ] < \| \nabla W \|^2 + \eta_0. 
\end{align}
Then, after a rotation and a translation of the space $\R^N$, there exist $(v_0,v_1) \in \energ$, a sign $\iota_0 \in \{ \pm 1\}$, 
a small real parameter $\ell$ and functions $x(t) \in \R^N, \lam(t) > 0$, defined for $t \in (0,T_+)$ such that, 
\begin{align*}
u(t) &= v_0 + \frac{\iota_0}{\lam(t)^{(N-2)/2}} W_\ell \left ( 0, \frac{\cdot - x(t)}{\lam(t)} \right ) + o_{\dot H^1}(1), \\
\p_t u(t) &= v_1 + \frac{\iota_0}{\lam(t)^{N/2}} \p_t W_\ell \left ( 0, \frac{\cdot - x(t)}{\lam(t)} \right ) + o_{L^2}(1), 
\end{align*}
and 
\begin{align*}
\lim_{t \rar T_+} \frac{\lam(t)}{T_+ - t} = 0, \quad \lim_{t \rar T_+} \frac{x(t)}{T_+ - t} = \ell \vec e_1, \quad 
|\ell | \leq C \eta_0^{1/4}.
\end{align*}
\end{thm}

\begin{rmk}\label{r84}
Theorem \ref{t83} is due to the authors in \cite{10}.  In two of the steps in the proof in \cite{10}, namely Proposition 3.1 and its Corollary 
3.2, and in Lemma 3.6, the false Pythagorean identities \eqref{sec39} and \eqref{sec310} are used, due to the fact that the factor 
in front of the space derivative and the $t$ derivative are different.  If one replaces condition \eqref{s83} by the \lq symmetric
condition'
\begin{align}\label{s84a}
\limsup_{t \rar T_+} \left [ \| \nabla u(t) \|^2 + \| \p_t u(t) \|^2 \right ] < \| \nabla W \|^2 + \eta_0, 
\end{align}
then one can use the correct Pythagorean identity \eqref{sec37} instead, and the proof given in \cite{10} is valid with only
minor modifications, for $N = 3$.  When $N = 5$, one encounters, using \eqref{s84a}, similar difficulties as those described in Remark \ref{rmk73}, and this leaves a gap in the proof, when $N = 5$.  We are now going to fill this gap, by giving a complete proof of Theorem \ref{t83}, where the proofs of Proposition 3.1 and its Corollary 3.2 in \cite{10} are accomplished using Lemma \ref{lem66}, and 
where the proof of Lemma 3.6 in \cite{10} is carried out through the use of the results in Section 6, combined with an argument in \cite{6}.  We will concentrate on the differences in the argument we are presenting here and that in \cite{10}.  
\end{rmk}

\begin{proof}[Proof of Theorem \ref{t83}]
We start by pointing out that, by \eqref{s83}, there is a single point in the singular set, by Corollary \ref{cor79}.  We thus
assume, without loss of generality, that $S = \{ 0 \}$.  We proceed in several steps. 


\emph{Step 1.}  We give the proof for a special sequence of times $\{ t_n \}_n$.  As before, we call $v(t)$ the regular part of the solution, and $a(t) = u(t) - v(t)$, so that (assuming, as we can, that $T_+(u) = 1$), we have 
$\supp( a(t), \p_t a(t)) \subseteq \{ |x| \leq 1 -t\}$, by Theorem \ref{thm75}, Definition \ref{defn76}, and \eqref{sec78s}.
We note that, since $v(t)$ is continuous at $t = 1$ in $\energ$,
\begin{align*}
\lim_{t \rar 1} \int_{|x| \leq 1-t} |\nabla_{t,x} v(t)|^2 dx = 0, 
\end{align*}
and thus, from \eqref{s83} we deduce 
\begin{align}\label{s84}
\limsup_{t \rar T_+} \int |\nabla a(t)|^2 + \frac{N-2}{2} |\p_t a(t)|^2 dx < \| \nabla W \|^2 + \eta_0.  
\end{align}
We now apply Lemma \ref{lem66} to $u$.  Thus, we find $\{ t_n \}_n$ such that \eqref{sec614} and \eqref{sec615} hold.  
From \eqref{sec614} and \eqref{s83}, we see that 
\begin{align*}
\int |\nabla Q_{\vl_j}^j (0) |^2 dx+ \frac{N-2}{2} \int |\p_t Q^j_{\vl_j}(0)|^2 dx < 
\| \nW \|^2 + \eta_0, 
\end{align*}
for $j = 1, \ldots, J_{\min}$.  We now apply the argument leading to \eqref{sec75} to conclude that, for $1 \leq j \leq J_{\min}$, 
we have
\begin{align}\label{s85}
\left ( 1 + \frac{|\vl_j|^4}{8} \right ) \int |\nabla Q^j |^2 dx < \| \nW \|^2 + \eta_0. 
\end{align}
As a consequence of \eqref{s85}, using Theorem \ref{thm21}, we conclude that, if $\eta_0$ is small, 
\begin{align}\label{s86}
Q^j = \pm W_{\lam_{0,j}} ( \cdot + x_{0,j} ).
\end{align}

At this point, we recall from Lemma \ref{lem66}, that $\left ( Q_{\vl_j}^j(0) , \p_t Q^j_{\vl_j}(0) \right ) = 
(U^j_L(0), \p_t U^j_L(0))$ where $t_{j,n} = 0$ for each $j = 1, \ldots, J_{\min}$.  We also note, from 
\eqref{s85} that, since $\int |\nabla Q^j|^2 dx = \int |\nabla W |^2 dx$, by \eqref{s86}, 
$\| \nabla W \|^2 |\vl_j|^4 \leq 8 \eta_0$.  An explicit computation shows that 
\begin{align}
\int |\nabla W_{\vl}(0)|^2 &= \int |\nabla W|^2 + O( |\vl|^2 ), \nonumber \\ 
\int |\p_t W_{\vl}(0)|^2 &= O( |\vl|^2 ) \int | \nW |^2 dx.  \label{s87}
\end{align}

We now apply \eqref{sec311} and \eqref{sec312}, with $\e= \eta_0$, and use the fact that $J_{\min}$ has an upper bound which
is fixed, since $J_{\min} \leq J_{\max}$, and $J_{\max}$ has one by the comments after \eqref{sec65}, and \eqref{s87}, 
\eqref{s85}, \eqref{s86}, and \eqref{s83}, to see that 
\begin{align*}
J_{\min} \| \nW \|^2 \leq \| \nW \|^2 + 3 \eta_0 + C J_{\min} \max_j |\vl_j|^2 \leq \| \nW \|^2 + 3 \eta_0 + C \eta_0^{1/2}.
\end{align*}
If $\eta_0$ is small enough, this shows that 
\begin{align}\label{s88a}
J_{\min} = 1, \quad Q^1 = \pm W_{\lam_0,1} ( \cdot + x_{0,1} ), \quad \vl_1 = \vl, \quad |\vl| \leq C \eta_0^{1/4}. 
\end{align}
We set $\lam_0 = \lam_{0,1}, x_0 = x_{0,1}$. 


Recall that $(u(t), \p_t u(t) ) \wa_{t \rar 1} (v_0,v_1)$ weakly in $\energ$, and that because of this, by Remark 
\ref{rmk311}, $(v_0,v_1)$ with scaling factor 1, $x$--translation term 0, and time translation term 0, appears in the profile decomposition for 
$(u(t_n), \p_t u(t_n)$.  Hence, since $v(t)$ is continuous in $\dot H^1 \times L^2$, 
$\left ( U^j_L, \{ \lam_{j,n}, x_{j,n}, t_{j,n} \}_n \right )_j$, for $U^j_L(0) \neq (v_0,v_1)$ is a profile decomposition of 
$(a(t_n), \p_t a(t_n))$ and hence, for $j$ such that $U^j_L(0) \neq (v_0,v_1)$, we have 
\begin{align}\label{s88}
|\lam_{j,n}| + | x_{j,n} | + |t_{j,n} | \leq C_j( 1 - t_n ), 
\end{align}
in  light of \cite{1}, page 154--155.  Notice also that if we have any sequence $\{ \tau_n \}_n$, $\tau_n \rar 1$ and a well--ordered 
profile decomposition of $\{ (u(\tau_n), \p_t u(\tau_n)) \}_n$, \eqref{s88} and the remarks preceding it hold.  Moreover, 
\begin{align}\label{sec89}
\left ( U^1_L, \{ \lam_{1,n}, x_{1,n}, t_{1,n} \}_n \right ) \prec \left ( (v_0,v_1), \{1, 0 , 0 \} \right ).
\end{align}
Indeed, if not, since $U^1_L$ does not scatter forward in time because $T_+(u) = 1$, we would need to have 
$\forall T < T_+(v(1), \p_t v(1))$, 
\begin{align}\label{s810}
\lim_{n \rar \infty} \frac{T - t_{1,n}}{\lam_{1,n}} < T_+(U^1).
\end{align}
But note that the evolution for $(v(1), \p_t v(1))$ starts at $t = 1$, so that $T_+(v(1), \p_t v(1)) > 0$.  Pick then 
$0 < T < T_+(v(1), \p_t v(1))$.  Since $U^1_L$ cannot be $(v_0,v_1)$, because if it were, by Proposition \ref{ppn318},
0 would not be a singular point, by \eqref{s88} we have, for $n$ large $\frac{T - t_{1,n}}{\lam_{1,n}} \simeq
\frac{1}{\lam_{1,n}} \rar_n \infty$, contradicting \eqref{s810}.  

Going back to our sequence $\{ t_n \}_n$ coming from Lemma \ref{lem66}, apply once more \eqref{sec311} and \eqref{sec312}
with $\e_0 = \eta_0$, use the form of $Q^1$, and the bound on $|\vl|$, as well as \eqref{s87}, and \eqref{s83}, to see that 
\begin{align*}
\int |\nabla v(t_n)|^2 dx + \frac{N-2}{2} \int |\p_t v(t_n)|^2 dx + \| \nabla W \|^2 \leq 
\| \nW \|^2 + 2 \eta_0 + C \eta_0^{1/2}.  
\end{align*}
Thus, for small $\eta_0$, the small data theory shows that $v$ scatters forward and backward in time and 
\begin{align*}
\sup_t \left [ \| \nabla v(t) \|^2 + \| \p_t v(t) \|^2 \right ] \leq C \eta_0^{1/2}.  
\end{align*}
We now use \eqref{s88} for $j = 1$, to see that 
by the continuity of $v(t)$ at $t = 1$ in $\energ$, $\lam_{1,n}^{N/2} \nabla_{t,x} v(t_n , \lam_{1,n} \cdot + x_{1,n} )
\wa (0,0)$ in $\energ$. Thus, combining this with Lemma \ref{lem66}, we conclude that $\lam_{1,n}^{N/2} \nabla_{t,x} a(t_n, 
\lam_{1,n} \cdot + x_{1,n} ) \wa_n \left ( \pm \nabla W_{\lam_0, \vl}(0), \pm \p_t W_{\lam_0, \vl}(0) \right )$.  Thus, after 
changing $\lam_{1,n} = \lam_n$ by a scaling factor, $x_{1,n} = x_n$ by a translation factor, and rotating $\R^N$ to transform 
$\vl_1 = \vl$ into $\ell \vec e_1$, and replacing $u$ by $-u$ if necessary, we have

\begin{cor}\label{c85}
Let $\tau = \{ t_n \}_n$, $t_n \rar 1$ be such that Lemma \ref{lem66} applies to $u$.  Then rotating the space variable and replacing $u$ by $-u$ if necessary, there exist $\lam_n$, $x_n$ such that 
\begin{align}\label{s811}
\left ( \lam_n^{(N-2)/2} a(t_n, \lam_n x + x_n), \lam_n^{N/2} \p_t a(t_n, \lam_n x + x_n ) \right ) 
\wa_{n \rar \infty} \left ( W_\ell(0,x), \p_t W_\ell(0,x) \right ), 
\end{align}
weakly in $\energ$ for some $\ell \in \R$, with 
\begin{align}\label{s812}
|\ell | \leq C \eta_0^{1/4}.
\end{align}
Furthermore, for large $n$, 
\begin{align}\label{s813}
\left \| \lam_n^{(N-2)/2} a(t_n, \lam_n \cdot + x_n ) - W_\ell(0, \cdot) \right \|^2_{\dot H^1} 
+ \frac{N-2}{2} \left \| \lam_n^{N/2} \p_t a(t_n, \lam_n \cdot + x_n ) - \p_t W_\ell(0, \cdot) \right \|^2_{L^2}  
\leq 2 \eta_0, 
\end{align}
and 
\begin{align}\label{s814}
\sup_{t \in \R} \left [ \| \nabla v(t) \|^2 + \| \p_t v(t) \|^2 \right ] \leq C \eta_0^{1/2}. 
\end{align}
\end{cor}

\begin{proof}[Proof of Corollary \ref{c85}]
\eqref{s811}, \eqref{s812}, and \eqref{s814} have already been established. \eqref{s813} follows from \eqref{s811}, 
\eqref{s83}, \eqref{s87}, and \eqref{s812}, using also \eqref{sec75}.  
\end{proof}

\begin{rmk}\label{r86}
We have also shown that $J_{\min} = 1$, and that for any well--ordered profile decomposition of $(u(\tau_n), \p_t u(\tau_n)), 
\tau_n \rar 1$, $(v_0,v_1)$ appears as one of the profiles, for a $j > 1$.  In addition the $\{ \lam_{j,n}, x_{j,n}, t_{j,n} \}$ verify, 
when the $j$th profile is not $(v_0,v_1)$, $|\lam_{j,n}| + |x_{j,n}| + |t_{j,n}| \leq C_j(1 - \tau_n)$. 
\end{rmk}


\begin{defn}\label{d86}
Let 
\begin{align*}
E_0 &= \lim_{t \rar 1} E(a(t), \p_t a(t)), \\
d_0 &= \lim_{t \rar 1} \int \nabla a(t) \p_t a(t) dx.
\end{align*}
\end{defn}

\begin{rmk}\label{r87}
Note that the fact that $(u(t), \p_t u(t) ) \wa_{t \rar 1} (v_0,v_1)$, $(v(t), \p_t v(t)) \rar_{t \rar 1} (v_0,v_1)$ 
shows that $E_0$ and $d_0$ are well defined, and $E_0 = E(u_0,u_1) - E(v_0,v_1)$ and 
$d_0 = \int \nabla u_0 u_1 dx - \int \nabla v_0 v_1 dx$, by the invariance of energy and momentum.  In view of 
\eqref{s813} and \eqref{s87}, we have: 
\begin{align}\label{s815}
|E_0 - \EW | + |d_0| \leq C \eta_0^{1/4}. 
\end{align}
\end{rmk}

The next item is to give further estimates on the parameters.  

\begin{lem}\label{l87}
The parameters $x_n$ and $\lam_n$ satisfy
\begin{align}
\lim_{n \rar \infty} \frac{\lam_n}{1 - t_n} &= 0, \label{s816}\\
\limsup_{n \rar \infty} \frac{|x_n|}{1 - t_n} &\leq C \eta_0^{1/4} \label{s817}.
\end{align}
\end{lem}


The proof of Lemma 3.3 in \cite{10} applies verbatim once we have Corollary \ref{c85} and \eqref{s88}.  The key 
ingredients are the fact that $W_\ell$ is not compactly supported and virial identities (Claim 2.11 in \cite{10}).  
The results obtained up to this point hold for $N = 3,4,5$.  The next step is fundamental, and a key ingredient in our approach, 
and uses crucially Proposition \ref{p81}.  

\begin{ppn}\label{p88}
Let $\{ t_n \}_n$ be as in Corollary \ref{c85}.  Then, for $N = 3,5$, 
\begin{align}\label{s818}
\lim_{n \rar \infty} \left (\lam_n^{(N-2)/2} a(t_n, \lam_n x + x_n), \lam_n^{N/2} \p_t a(t_n, \lam_n x + x_n ) \right )
= \left ( W_\ell(0), \p_t W_\ell(0) \right ), 
\end{align}
strongly in $\energ$. 
\end{ppn}

The proof of Proposition \ref{p88} is verbatim the one in Proposition 3.4 of \cite{10}.  It uses Lemma \ref{l87}.  The key idea in the proof
of it is that, if \eqref{s818} does not hold, because of Proposition \ref{p81}, energy will be sent outside the light cone at time 
$t = 1$, contradicting the support of $a(t)$ (case $t \geq 0$ in \eqref{s81}), or arbitrarily close to the boundary of the light cone
at time $0$ (case $t \leq 0$ in \eqref{s82}), which is also a contradiction.  

\begin{cor}\label{c89}
We have 
\begin{align*}
E_0 &= E ( W_\ell(0), \p_t W_\ell(0)), \\
d_0 &= -E_0 \ell \vec e_1 = \int \nabla W_\ell(0) \p_t W_\ell(0) dx. 
\end{align*}
\end{cor}

Corollary \ref{c89} follows immediately from Proposition \ref{p88}.  

To obtain convergence now for all sequences of time, and not just the sequence $\{ t_n \}_n$ from Lemma \ref{lem66}, 
we use an argument from \cite{6}.  

\emph{Step 2.} In this step we show that $J_{\max} = 1$ and $\mathcal E_{\min} = E(W_\ell(0), \p_t W_\ell(0))$ which we formulate as a lemma.

\begin{lem}\label{l810}
Let $J_{\max}$ be as in \eqref{sec66}.  Then $J_{\max} = 1$, and $\mathcal E_{\min} = E(W_\ell(0), \p_t W_\ell(0))$, 
where $\mathcal E_{\min}$ is defined in \eqref{sec69}.
\end{lem}

\begin{proof}
Let $\{ t_n \}_n$ be as before, and consider the well--ordered profile decomposition corresponding to it.   We have seen that 
$(U^1_L(0), \p_t U^1_L(0)) = ( W_\ell(0), \p_t W_\ell(0))$.  By Lemma \ref{lem66}, $\tau = \{ t_n \}_n \in \mathcal S_3
\subset \mathcal S_2 \subset \mathcal S_1$ (see the definitions in Section 6).  Thus, we have $J_{\max}$ non--scattering
profiles and $\mathcal E_{\min} = \sum_{j = 1}^{J_{\max}} E(U^j, \p_t U^j)$.  We first notice that each non--zero 
nonlinear profile has strictly positive energy, that $E(a(t), \p_t a(t)) \geq 0$ for $t$ near 1, that $E(v_0,v_1) \geq 0$, that for $J$ large, $n$ large $E(w^J_{0,n}, w_{1,n}^J ) \geq 0$ and that, also for $J$ large, 
\begin{align*} 
\lim_{n \rar \infty} E(w^J_{0,n}, w^J_{1,n} ) = 0 \implies \lim_{n \rar \infty} \| (w^J_{0,n}, w^J_{1,n} ) \|_{\energ} = 0.
\end{align*}
Indeed, by the variational estimate Lemma \ref{lem24}, if $\eta_0$ is small, for $t$ near 1 $E(a(t),\p_t a(t)) \geq 0$ by \eqref{s83} and 
$E(v_0,v_1) \geq 0$ by \eqref{s814}.  Because of \eqref{sec311}, taking $J$ large, and then $n \rar \infty$, we see that, for 
the core $j$, we have, if $\eta_0$ is small, that $E(U^j, \p_t U^j) \geq 0$, and it is 0 only if 
$(U^j, \p_t U^j) = (0,0)$.  This is because the variational estimate in Corollary \ref{cor25} gives, for $\eta_0$ small, that $E(U^j(0), 0) \geq 0$, and 
if $U^j(0) \neq 0$, $E(U^j(0), 0) > 0$.  If $j$ is scattering, the argument in the proof of Lemma \ref{lem44}, shows that 
\begin{align*}
E  \left ( U^j_L(-t_{j,n} / \lam_{j,n}, \p_t U^j_L(-t_{j,n} / \lam_{j,n} ) \right ) = 
\frac{1}{2} \| \nabla U^j_L(0) \|^2 + \frac{1}{2} \| \p_t U^j_L(0) \|^2 + o_n(1),
\end{align*}
which, by definition of nonlinear profile and invariance of the nonlinear energy gives the assertion.  The argument for 
$(w^J_{0,n}, w^J_{1,n} )$ is similar to the one for core $j$.  By \eqref{sec311}, taking $J$ large, and using the variational 
estimate Lemma \ref{lem24}, $E(w^J_{0,n}, 0) \geq 0$, and $E(w^J_{0,n}, 0)$ tends to 0 only if $\| \nabla w^J_{0,n} \|$ tends to 0 by Corollary \ref{cor25}.    
Thus, if $E(w^J_{0,n}, w^J_{1,n})$ tends to 0, $E(w^J_{0,n}, 0) \geq 0$ must also tend to 0, and hence $\| \nabla w^J_{0,n} \|$ tends to 0, but then 
$E(w^J_{0,n}, w^J_{1,n} ) = \frac{1}{2} \| w^J_{1,n} \|^2 + o_n(1)$, so that $\| w^J_{1,n} \|$ tends to 0.  

With this statement in hand, and recalling from Remark \ref{r87} and Corollary \ref{c89} that 
\begin{align*}
E(u_0,u_1) = E(v_0,v_1) + E( W_\ell(0), \p_t W_\ell(0) ), 
\end{align*}
using the Pythagorean expansion from the nonlinear energy, which follows from \eqref{sec37} and \eqref{sec38}, and the definition of the nonlinear profile, 
\begin{align*}
E(u_0,u_1) = \sum_{j = 1}^{J_{\max}} E(U^j, \p_t U^j) + \sum_{j = J_{\max} + 1}^J E(U^j, \p_t U^j) 
+ E(w^J_n(0), \p_t w^J_n(0)) + o_n(1). 
\end{align*}
Since $(v_0,v_1)$ is one of the scattering profiles, we have that the right hand side equals
\begin{align*}
E(W_\ell(0), \p_t W_\ell(0) ) &+ E(v_0,v_1) \\ &  
\sum_{j = 2}^{J_{\max}} E(U^j, \p_t U^j) + \sum_{{\substack{j = J_{\max} + 1 \\ U^j \neq v}}}^J E(U^j, \p_t U^j)
+ E(w^J_n(0), \p_t w^J_n(0)) + o_n(1). 
\end{align*}
We then obtain 
\begin{align*}
0 = 
\sum_{j = 2}^{J_{\max}} E(U^j, \p_t U^j) + \sum_{{\substack{j = J_{\max} + 1 \\ U^j \neq v}}}^J E(U^j, \p_t U^j) 
+ E(w^J_n(0), \p_t w^J_n(0)) + o_n(1),
\end{align*}
which gives, in light of the nonnegativity of all the terms that $(U^j, \p_t U^j), 2 \leq j \leq J_{\max}$ are all 0, and hence
$J_{\max} = 1$, and since $\tau \in \mathcal S_2$, $\mathcal E_{\min} = E(W_\ell(0), \p_t W_\ell(0))$.  
\end{proof}

\emph{Step 3.}  In this step we establish convergence for all times and conclude the proof of Theorem \ref{t83}. 

\begin{lem}\label{l811}
Given any $\tau_n \rar 1$, there exist $\lam_n > 0, x_n \in \R^N$ such that 
\begin{align*}
\left (\lam_n^{(N-2)/2} a(\tau_n, \lam_n \cdot + x_n), 
\lam_n^{N/2} \p_t a(\tau_n, \lam_n \cdot + x_n) \right )
\end{align*}
has a convergent subsequence in $\energ$. 
\end{lem}

\begin{proof}
Let $\tilde \tau = \{ \tau_n \}_n$,  After extraction, and reordering the profiles, by Claim \ref{clm316}, $\tilde \tau \in \mathcal S_0$.  We claim that, in the well--ordered profile decomposition of $(u(\tau_n), \p_t u(\tau_n))$, all the profiles which scatter 
forward in time (other than $(v_0,v_1)$) must be 0, and $(w^J_n(0), \p_t w^J_n(0)) \rar_n (0,0)$ in $\energ$. In fact, we know, in view of the fact that $J_{\max} = 1$, and that the first profile cannot scatter forward in time, since $0 \in S$, that $U^1$
does not scatter forward and $U^j$, for $j \geq 2$ do. Hence $\tilde \tau_1 \in \mathcal S_1$.  Notice that, by the argument given in the proof of Lemma \ref{l810}, all the non--zero nonlinear profiles have positive energy, as does $(w^J_n(0), \p_t w^J_n(0))$ for
$J$ large, with $E(w^J_n(0), \p_t w^J_n(0))$ going to 0 if and only if $(w^J_n(0), \p_t w^J_n(0))$ goes to 0 in $\energ$.  Hence by \eqref{sec37}, \eqref{sec38}, 
\begin{align*}
E(u_0,u_1) &= E \left ( W_\ell(0), \p_t W_\ell(0) \right ) + E(v_0,v_1) \\
&= E(U^1, \p_t U^1) + \sum_{j =2}^{J} E(U^j, \p_t U^j) + E(w^J_n(0), \p_t w^J_n(0)) + o_n(1) \\
&\geq \mathcal E_{\min} + E(v_0,v_1) + 
\sum_{{\substack{ j =2 \\ U^j \neq v}}}^{J} E(U^j, \p_t U^j) + E(w^J_n(0), \p_t w^J_n(0)) + o_n(1) \\
&= E \left ( W_\ell(0), \p_t W_\ell(0) \right ) + E(v_0,v_1) +
\sum_{{\substack{ j =2 \\ U^j \neq v}}}^{J} E(U^j, \p_t U^j) + E(w^J_n(0), \p_t w^J_n(0)) + o_n(1).
\end{align*}
This establishes the claim. 

We know that $U^1$ does not scatter forward in time.  We claim next that $U^1$ does not scatter backward in time.  If it did, 
by the backward in time version of Theorem \ref{thm312}, for $n$ large and $t < 0$, 
\begin{align*}
\vec u(\tau_n + t) &= \vec v(\tau_n + t) + \vec U^1_n(t) + \vec w^J_{L,n}(t) + \vec r^J_n(t),
\end{align*}
where both $\| \vec w_{L,n}^J(t) \|_{\energ}$ and $\sup_{t<0} \| \vec r_n^J(t) \|_{\energ}$ tend to 0 with $n$. Note that since
$U^1$ does not scatter forward, there exists $\de_0 > 0$, given by the local theory of the Cauchy problem, so that 
$\| (U^1(t), \p_t U^1(t) ) \|_{\energ} \geq \de_0$, for all $t$.   Choose $t = t_0 - \tau_n$, where $0 < t_0 < 1$ is fixed.  Then 
\begin{align*}
\vec u(t_0) = \vec v (t_0) + \vec U^1_n(t_0 - \tau_n) + o_n(1),
\end{align*}
which is a nontrivial profile decomposition for the fixed function $\vec u(t_0) - \vec v(t_0)$. This means that $\lam_{1,n} \equiv 1$.  
But by Remark \ref{r86}, $\lam_{1,n} \leq C(1 - \tau_n)$, a contradiction.  Since $U^1$ does not scatter forward or backward,
$\left |\frac{t_{1,n}}{\lam_{1,n}} \right | \leq C_0$, so that, after changing $U^1_L$ be a time translation, we can assume that $t_{1,n} \equiv 0$.  Thus, 
\begin{align*}
\vec a(t_n) = \left ( \frac{1}{\lam_{1,n}^{(N-2)/2}} U^1 \left (0 , \frac{\cdot - x_{1,n}}{\lam_{1,n}} \right ), 
\frac{1}{\lam_{1,n}^{N/2}} \p_t U^1 \left (0 , \frac{\cdot - x_{1,n}}{\lam_{1,n}} \right )  \right ) + o_n(1),
\end{align*}
which gives our conclusion. 
\end{proof}

Once Lemma \ref{l811} is established, the proof of Theorem \ref{t83} follows a standard path.  As in Lemma 8.5 in 
\cite{9}, if $\left (\lam_n^{(N-2)/2} a(\tau_n, \lam_n \cdot + x_n), 
\lam_n^{N/2} \p_t a(\tau_n, \lam_n \cdot + x_n) \right ) \rar_n (U_0,U_1)$, strongly in $\energ$, Lemma \ref{l811} implies that 
the solution $U$ of \eqref{nlw} with data $(U_0,U_1)$ is compact.   Because of \eqref{s83} and Theorem 
\ref{thm57}, 
\begin{align*}
U = \frac{\pm}{\lam_0^{(N-2)/2}} W_{\ell'} \left ( \frac{t}{\lam_0}, \frac{ \mathcal R(x) - x_0 }{\lam_0} \right ), 
\end{align*}
for $\mathcal R$ a rotation of $\R^N$.  From this, and $d_0 = -E_0 \ell \vec e_1$, since $d_0 = -E_0 \ell' \mathcal R(\vec e_1)$, 
we have that $\mathcal R$ is a rotation with axis $(0,\vec e_1)$, and $\ell = \ell'$.  Hence, as $W_\ell$ is invariant under this kind 
of rotation, 
\begin{align*}
\frac{1}{\lam_0^{(N-2)/2}} W_{\ell} \left ( \frac{t}{\lam_0}, \frac{x - x_0 }{\lam_0} \right )
= \frac{1}{\lam_0^{(N-2)/2}} W_{\ell'} \left ( \frac{t}{\lam_0}, \frac{ \mathcal R(x) - x_0 }{\lam_0} \right ).  
\end{align*}
(This argument is in Lemma 3.6 of \cite{10}.)   Finally, a continuity argument as in Corollary 3.7 of \cite{10} finishes the proof of 
Theorem \ref{t83}.
\end{proof}

\begin{cor}\label{c812}
The translation parameter $x(t)$ in Theorem \ref{t83} verifies
\begin{align*}
\lim_{t \rar 1} \frac{x(t)}{1 - t} = \ell \vec e_1. 
\end{align*}
\end{cor}
 
This is proved in Lemma 3.8 of \cite{10}.

\end{document}